\documentclass[11pt,letterpaper]{article}
\usepackage{amsfonts, amsmath, amssymb, amscd, amsthm, color, graphicx, mathabx, mathrsfs, wasysym, setspace, mdwlist, calc, float, mathtools, tikz-cd}
\usepackage{setspace}
\usepackage{comment}

\usepackage{enumitem}
\usepackage{tocloft}
\usepackage{xcolor}

\usepackage{hyperref}
\hypersetup{linktocpage}

\hypersetup{colorlinks,
    linkcolor={red!50!black},
    citecolor={blue!80!black},
    urlcolor={blue!80!black}}
\usepackage{float}

\numberwithin{equation}{section}

\usepackage{bbm}

\renewcommand{\phi}{\varphi}
\newcommand{\WR}{\mathcal{WR}}

\newcommand{\C }{\mathcal C}

\renewcommand{\P }{\mathcal P}
\renewcommand{\kappa }{\varkappa}

\renewcommand{\wr}{{ \rm wr }}

\newcommand{\M}{\mathcal M}

\newcommand{\ra}{\rightarrow}
\newcommand{\ca}{\curvearrowright}
\newcommand{\A}{\mathcal A}

\newcommand{\D}{\mathcal D}

\newcommand{\Q}{\mathcal Q}
\newcommand{\R}{\mathcal R}

\newcommand{\W}{\mathcal W}

\newcommand{\sP}{\mathscr P}
\newcommand{\sU}{\mathscr U}
\newcommand{\sN}{\mathscr N}

\newcommand{\sZ}{\mathscr Z}
\newcommand{\sR}{\mathscr R}

\newcommand{\sM}{\mathscr M}

\newcommand{\La}{\Lambda}

\newcommand{\oo}{\bar\otimes}
\newtheorem{thm}{Theorem}[section]
\newtheorem*{thm*}{Theorem}
\newtheorem{cor}[thm]{Corollary}
\newtheorem{lem}[thm]{Lemma}
\newtheorem{claim}[thm]{Claim}
\newtheorem{prop}[thm]{Proposition}

\theoremstyle{definition}
\newtheorem{defn}[thm]{Definition}

\theoremstyle{remark}
\newtheorem{rem}[thm]{Remark}

\let\OLDthebibliography\thebibliography
\renewcommand\thebibliography[1]{
  \OLDthebibliography{#1}
  \setlength{\parskip}{1.5pt}
  \setlength{\itemsep}{1.5pt plus 0.3ex}
}

\begin{document}


\title{Non-amenable C$^*$-superrigid groups that are not W$^*$-superrigid}

\author{Juan Felipe Ariza Mejía, Ionu\c{t} Chifan, Adriana Fernández Quero}

\date{}

\maketitle

\begin{abstract}
\noindent
Using techniques at the intersection of deformation/rigidity theory, geometric group theory, and the theory of $C^*$-algebras, we construct a continuum of nonamenable groups $G$ that can be completely reconstructed from their reduced $C^*$-algebras $C_r^*(G)$, but not from their group von Neumann algebras $\mathrm{L}(G)$. These groups arise as infinite direct sums of amalgamated free product groups and constitute the first known examples of nonamenable groups exhibiting this phenomenon. 

In addition, we provide examples of finite direct products of amalgamated free product groups that are simultaneously $C^*$-superrigid and $W^*$-superrigid. Finally, for a fairly large subclass of these amalgamated free product groups $G$, we show that all $\ast$-endomorphisms of $C_r^*(G)$ are weakly inner.
\end{abstract}

\vspace{3mm}



\section{Introduction}


Let $G$ be a countable discrete group, and denote by $\mathbb{C}[G]$ its group ring. Concretely, $\mathbb{C}[G]$ consists of finite linear combinations of elements of $G$, and acts on the Hilbert space $\ell^2G$ of square-summable sequences on $G$ via the left regular representation $u: G \to \mathcal{B}(\ell^2G)$, defined by $(u_g \xi)(h) = \xi(g^{-1}h)$ for every $g,h \in G,\:\xi \in \ell^2G$. The group ring $\mathbb{C}[G]$ can thus be identified with the $\ast$-subalgebra of $\mathcal{B}(\ell^2G)$ generated by $\{u_g \mid g \in G\}$. Two natural operator-algebraic completions of $\mathbb{C}[G]$ arise from this representation, yielding the \emph{reduced group $C^*$-algebra} $C_r^*(G)$, and the \emph{group von Neumann algebra} $\text{L}(G)$ of $G$. Their relationship can be expressed as $\mathbb{C}[G] \;\subseteq\; C_r^*(G) \;\subseteq {\rm L}(G)$, where $C_r^*(G)$ is the closure of $\mathbb{C}[G]$ in the operator norm on $\mathcal{B}(\ell^2G)$, and ${\rm L}(G)$ is the closure of $\mathbb{C}[G]$ in the weak operator topology on $\mathcal{B}(\ell^2G)$.

As observed by Murray and von Neumann in \cite{MvN43}, the group von Neumann algebra ${\mathrm L}(G)$ has trivial center (and is therefore a $\mathrm{II}_1$ factor) if and only if all nontrivial conjugacy classes of $G$ are infinite. In this case, $G$ is called an \emph{icc} (infinite conjugacy class) group. 
This situation, in which ${\mathrm L}(G)$ behaves as a ``simple'' object, is the most interesting one to study, and it will be the prevailing setting considered in this paper.

A fundamental question is to understand to what extent the original group $G$ can be completely reconstructed from the algebraic categories naturally associated with it, namely the group ring $\mathbb{C}[G]$, the reduced group $C^*$-algebra $C_r^*(G)$, and the group von Neumann algebra ${\mathrm L}(G)$. 
More precisely, we recall the following rigidity notions. A group $G$ is called \emph{ring-superrigid} if for every group $H$ and every $*$-isomorphism $\mathbb{C}[G] \cong \mathbb{C}[H]$ one necessarily has $G \cong H$, \emph{$C^*$-superrigid} if for every group $H$ and every $*$-isomorphism $C_r^*(G) \cong C_r^*(H)$ one necessarily has $G \cong H$; and \emph{$W^*$-superrigid} if for every group $H$ and every $*$-isomorphism ${\text L}(G) \cong {\text L}(H)$ one necessarily has $G \cong H$.

Even in the purely algebraic setting of group rings, identifying superrigid groups is a highly nontrivial problem.  For instance, if $G_1$ and $G_2$ are finite abelian groups with $|G_1| = |G_2| = k$, then there are canonical isomorphisms $\mathbb{C}[G_1] \cong \mathbb{C}[G_2] \cong \mathbb{C}^k$. In particular, no finite abelian group of composite order is ring-superrigid. At the opposite end of the spectrum, one can easly see that a group $G$ is ring-superrigid whenever $\mathbb{C}[G]$ satisfies Kaplansky’s unit conjecture. 
However, the full extent of ring-superrigidity remains unknown, especially in light of the intricate counterexamples to Kaplansky’s conjecture that have emerged in recent years, \cite{Ga21,Mu21}.

Since the completions $C_r^*(G)$ and ${\mathrm L}(G)$ are taken with respect to progressively weaker topologies (with the weak operator topology being, in a sense, the weakest natural topology one may impose on operators), the corresponding reconstruction problems become increasingly subtle. As the resulting algebras grow larger, they tend to retain progressively less information about the original group structure, making rigidity phenomena substantially more difficult to detect.

When $G$ is amenable, the group von Neumann algebra ${\rm L}(G)$ typically retains very little information about the underlying group structure. This phenomenon is well illustrated by the following classical examples. First, if $G$ is an infinite abelian group, then standard results in von Neumann algebra theory imply that ${\rm L}(G)$ is always isomorphic to the abelian von Neumann algebra $L^\infty([0,1],\lambda)$, where $\lambda$ denotes the Lebesgue measure. In particular, in this case ${\rm L}(G)$ depends only on the fact that $G$ is infinite and abelian, and not on any finer algebraic properties of $G$. More strikingly, Connes' celebrated theorem asserts that for any icc amenable group $G$, the associated group von Neumann algebra ${\rm L}(G)$ is isomorphic to $\mathcal{R}$, the unique hyperfinite $\mathrm{II}_1$ factor in the sense of Murray and von Neumann \cite{MvN43}. Consequently, aside from amenability (and the icc condition), ${\rm L}(G)$ does not encode any additional algebraic information about $G$. Thus, in both the abelian and the icc amenable settings, the passage from a group to its von Neumann algebra results in a dramatic loss of group-theoretic information.

In the $\mathrm{C}^*$-algebraic setting, the situation is radically different. A classical classification result from \cite{Sc74} for homogeneous $\mathrm{C}^*$-algebras implies that every torsion-free infinite abelian group $G$ is $\mathrm{C}^*$-superrigid; this depicts a sharp contrast with the case of abelian group von Neumann algebras discussed before. Moreover, in recent years additional examples of amenable $\mathrm{C}^*$-superrigid groups have been discovered. These include certain torsion-free virtually abelian groups in \cite{CKRTW17} and \cite{CLMW25}, as well as more structurally complex classes like two-step nilpotent groups \cite{ER18} and free nilpotent groups \cite{Om18}. These results demonstrate that, although amenability often weakens rigidity properties, the reduced group $\mathrm{C}^*$-algebra can still retain enough information to recover the underlying group in a variety of amenable settings. One should mention at this time we do not know if there exists an icc amenable $C^*$-superrigid group. 

In the nonamenable case, the situation is considerably more complex, both at the level of $C^*$-algebras and von Neumann algebras. Finding examples of $W^*$-superrigid nonamenable groups has been an active research direction over the past decades and has been studied almost exclusively through methods in deformation/rigidity theory. The first class of $W^*$-superrigid groups was discovered by Ioana, Popa, and Vaes \cite{IPV10}, where a large class of generalized wreath product groups was shown to have this property. Subsequently, additional examples of $W^*$-superrigid groups arising from coinduced groups, amalgamated free products, tree groups, and semidirect products with nonamenable cores were found in \cite{BV13, Ber14, CI17, CD-AD20, CD-AD21}. Some of these examples are also $C^*$-superrigid \cite{CI17, CD-AD20, CD-AD21}. 

Although it is natural to expect that there should exist a broader variety of nonamenable $C^*$-superrigid groups than $W^*$-superrigid ones, relatively little is known in this direction. In particular, it was previously unknown whether there exist nonamenable groups that are $C^*$-superrigid but not $W^*$-superrigid. One of the main goals of this paper is to address this question by providing a large class of such examples.

\begin{thm}\label{Thm:main}
There exists a family of countable groups $\{G_i\}_{i \in I}$ with $|I| = 2^{\aleph_0}$ such that $G_i \not\cong G_j$ whenever $i \neq j$, and each group $G_i$ is $C^*$-superrigid but not $W^*$-superrigid.
\end{thm}

\noindent The groups $G_i$ arise as infinite direct sums $G_i = \bigoplus_{n} K_n$, where each summand $K_n$ is an amalgamated free product of wreath-like product groups with property (T), constructed in \cite{CIOS1,CIOS5}. The proof of Theorem \ref{Thm:main} builds on and extends several techniques developed in earlier works, including \cite{CdSS15,CI17,CU18,CIOS1,AMCOS25,CIOS5}.

\subsection{Main results}\label{ATdef}

Before stating our main results, we consider the following class of groups.

\begin{defn}\label{def:AT} Denote by $\mathcal{AT}$ the collection of all amalgamated free product groups $G = G_1 \ast_C G_2$ satisfying the following conditions:

\begin{enumerate}
    \item[(i)] For $j=1,2$, the group $G_j\in\WR(A_j,B_j \curvearrowright I_j)$ is a wreath-like product with property (T); see Definition \ref{wlp}.
    \item[(ii)] Each $A_j$ is a nontrivial abelian group, and each $B_j$ is a nontrivial icc subgroup of a hyperbolic group.
    \item[(iii)] The action $B_j \curvearrowright I_j$ is such that, for every $i \in I_j$, the stabilizer $\mathrm{Stab}_{B_j}(i)$ is amenable.
    \item[(iv)] For every nontrivial element $b\in B_j$ its centralizer $C_{B_j}(b)$ is amenable.   
    \item[(v)] The common subgroup $C < G_1,G_2$ is icc, nonamenable, with Haagerup property, and it is almost malnormal in both $G_1$ and $G_2$. 
\end{enumerate}
\end{defn}

\noindent Many examples of groups satisfying the prior conditions have already been constructed using deep group-theoretic methods, such as Dehn filling \cite{Osi06,DGO11} and the Cohen--Lyndon property \cite{Sun}; see, for instance, \cite{CIOS1,CIOS3,CIOS4}. Many of these groups naturally contain malnormal free subgroups, as required by condition (v). In fact, the class $\mathcal{AT}$ is quite vast, containing continuum many elements.

Our $C^*$-superrigidity results rely heavily on corresponding $W^*$-superrigidity phenomena for amalgamated free products arising from groups in the class $\mathcal{AT}$. For this reason, we begin by presenting our rigidity results in the von Neumann algebraic setting. 

The first examples of $W^*$-superrigid amalgamated free products were obtained in \cite{CI17}, where the authors studied free products of product groups amalgamated over amenable subgroups. Subsequent developments further expanded this framework; in particular, \cite{CD-AD20,CD26} produced new examples allowing amalgamation over nonamenable subgroups. 

In the present paper, we further extend this line of research by adding the groups in class $\mathcal{AT}$ to the list of known W$^*$-superrigd amalgamated free product groups. These examples give rise to non--property~(T), solid $\mathrm{II}_1$ factors, thereby enriching the class of known $W^*$-superrigid von Neumann algebras and providing new sources of rigidity phenomena beyond the previously studied settings.

\begin{thm}\label{ATrigid} Let $G\in \mathcal{AT}$. Let $H$ be an arbitrary group and assume that $\theta \colon \mathrm{L}(G)^t \to \mathrm{L}(H)$ is a $*$-isomorphism, for some $0<t\leq 1$. Then $t=1$ and $G\cong H$. Moreover, there exist a group isomorphism $\delta \colon G \to H$, a character $\eta \colon G \to \mathbb{T}$, and a unitary $w \in \mathrm{L}(H)$ such that
\[
\theta(u_g) = \eta(g)  w v_{\delta(g)} w^*, \text{ for all } g \in G.
\]
\end{thm}

\noindent By recycling earlier techniques on product rigidity developed in \cite{CdSS15,CdSS17,CU18}, we also establish that the groups in the class $\mathcal{AT}$ satisfy this rigidity phenomenon. More precisely, given groups $G_1, \ldots, G_n \in \mathcal{AT}$, their direct product $G = G_1 \times \cdots \times G_n$ retains the product feature in the von Neumann algebraic framework: \emph{any group $H$ satisfying $\mathrm{L}(G) \cong \mathrm{L}(H)$, up to amplifications, must admit a direct product decomposition $H = H_1 \times \cdots \times H_n$ such that, up to amplification, one also has $\mathrm{L}(H_i) \cong \mathrm{L}(G_i)$ for each $1 \leqslant i \leqslant n$.} 

Combining this product rigidity result with the rigidity established above, we obtain new $W^*$-superrigidity results for direct products of groups in the class $\mathcal{AT}$.

\begin{thm}\label{superr} For every $1\leq i\leq n$, let $G_i\in\mathcal{AT}$, set $G=G_1 \times \cdots \times G_n$. Let $H$ be an arbitrary group and assume that $\theta \colon {\rm L}(G)^t\ra {\rm L}(H)$ is a $\ast$-isomorphism, for some $0<t\leq 1$. Then $t=1$ and $G \cong H$. In addition, there exist a group isomorphism $\delta \colon G \to H$, a character $\eta \colon G \to \mathbb{T}$, and a unitary $w \in \mathrm{L}(H)$ such that 

\begin{equation*}
    \theta(u_g) = \eta(g)  w v_{\delta(g)} w^*, \text{ for all } g \in G.
\end{equation*}
\end{thm}

\noindent However, if one considers infinite direct sums $G = \bigoplus_{n \in \mathbb{N}} G_n$ with $G_n \in \mathcal{AT}$, one immediately observes that such groups are never $W^*$-superrigid. Indeed, in this case one has ${\rm L}(G) \cong \bar\otimes_{n \in \mathbb{N}} \mathcal M_n$, where $\mathcal M_n = {\rm L}(G_n)$ and it is well known that infinite tensor products of II$_1$ factors are always McDuff factors \cite{Mc69}. To see this, first note that, using matrix units and amplifications, one always has $\mathcal M_n \cong \mathcal M_n^{1/2} \bar\otimes \mathbb M_2(\mathbb C)$. Moreover, if $\mathcal R$ denotes the Murray--von Neumann hyperfinite II$_1$ factor \cite{MvN43}, then $\mathcal R \bar\otimes \mathcal R \cong \mathcal R$. Combining these observations, we obtain
\begin{equation}\begin{split}
\bar\otimes_{n\in \mathbb N} \mathcal M_n 
&\cong \bar \otimes_{n\in\mathbb N} (\mathcal M_n ^{1/2}\otimes \mathbb M_2(\mathbb C))
\cong (\bar\otimes_{n\in \mathbb N} \mathcal M_n^{1/2})\bar\otimes (\bar\otimes_{n\in \mathbb N} \mathbb M_2(\mathbb C))\\
&\cong (\bar\otimes_{n\in \mathbb N} \mathcal M_n^{1/2})\bar\otimes \mathcal R
\cong(\bar\otimes_{n\in \mathbb N} \mathcal M_n^{1/2})\bar\otimes (\mathcal R \bar \otimes \mathcal R)\\
&\cong((\bar\otimes_{n\in \mathbb N} \mathcal M_n^{1/2})\bar\otimes \mathcal R) \bar \otimes \mathcal R
\cong (\bar\otimes_{n\in \mathbb N} \mathcal M_n) \bar\otimes \mathcal R.
\end{split}\end{equation}

\noindent Thus, we conclude that ${\rm L}(G) \cong {\rm L}(G)  \bar{\otimes}  \mathcal{R}$. Since, by \cite{Co76}, the hyperfinite factor $\mathcal R\cong {\rm L}(A)$, for any icc amenable group $A$, it follows that
${\rm L}(\bigoplus_{n\in\mathbb N} G_n) \cong {\rm L}\left ((\bigoplus_{n\in\mathbb N} G_n) \times A\right )$.
On the other hand, one easily checks that for every icc group $A$, the groups $\bigoplus_{n\in\mathbb N}G_n$ and $ (\bigoplus_{n\in\mathbb N} G_n)\times A$ are never isomorphic as each $G_n$ has trivial amenable radical.

Remarkably, despite this, the torsion-free groups in class $\mathcal{AT}$ still satisfy strong rigidity phenomena at the level of their reduced $C^*$-algebras. More precisely, by combining von Neumann algebraic deformation/rigidity techniques with several $C^*$-algebraic properties—such as uniqueness of the trace and the absence of nontrivial projections, as ensured by groups satisfying the Baum-Connes conjecture—we obtain the following $C^*$-superrigidity result. 
\begin{thm}\label{superr2}
For any $n\in \mathbb N$, let $G_n \in \mathcal{AT}$ be any torsion free group and consider the infinite direct sum $G = \bigoplus_n G_n$. Let $H$ be an arbitrary countable group and $\theta \colon \mathrm{C}^*_r(G) \to \mathrm{C}^*_r(H)$ a $*$-isomorphism. Then, $G \cong H$. Moreover, there exist a group isomorphism $\delta \colon G \to H$, a character $\eta \colon G \to \mathbb{T}$, and an automorphism $\psi \in Aut(\mathrm{C}_r^*(H))$ 
satisfying

\begin{equation*}
    \theta(u_g) = \psi (\eta(g)v_{\delta(g)}),\text{ for all } g \in G.
\end{equation*}
\end{thm}
\noindent In some cases, one can pick the automorphism $\psi \in \text{Aut}(C_r^*(H))$ to be approximately weakly inner---there is a sequence of unitaries $(w_n)_n \subset {\rm L}(H)$ with $\|\psi(x)w_n -w_n x\|_\infty\ra 0$ for all $x\in C_r^*(H)$. 

Notice that Theorem \ref{superr2} highlights once again how much ``smaller'' the C$^*$-algebras of groups are compared with their von Neumann algebraic counterparts. It would be interesting to see if the groups in the previous theorem satisfy the McDuff rigidity phenomenon discovered recently in \cite{AMCOS25}. If so, the $W^*$-superrigidity would only be missed by an ``amenable room'' in the ``tail'' of the infinite tensor product. 

We note that these groups, together with those independently considered by Curda and Drimbe, provide the first known examples of non-amenable product groups that are C$^*$-superrigid \cite{CD26}. In their work, the authors treat only finite products; however, unlike our Theorem \ref{superr}, in \cite[Theorem A]{CD26} the authors were able to establish W$^*$-superrigidity for \emph{arbitrary} amplifications at the level of the associated von Neumann algebras.

The methods developed to prove the superrigidity result in Theorem \ref{ATrigid} can also be successfully combined with $C^*$-algebraic techniques on the uniqueness of the trace to describe all endomorphisms of the $C^*$-algebras of groups in the class $\mathcal{AT}$. These results complement previous studies of the symmetries of reduced group $C^*$-algebras \cite{CI17,CD-AD20,CD-AD21}.

\begin{thm}\label{ATembeddings} 
Let $G \in \mathcal{AT}$ and let $\theta \colon C_r^*(G) \to C_r^*(G)$ be any $*$-endomorphism. Then there exist a group monomorphism $\delta \colon G \to G$, a character $\eta \colon G \to \mathbb{T}$, and a unitary $w \in \mathrm{L}(G)$ such that
\[
\theta(u_g) = \eta(g)  w v_{\delta(g)} w^*, \quad \text{for all } g \in G.
\]
\end{thm}

\noindent A group $G = G_1 \ast_C G_2 \in \mathcal{AT}$ is said to belong to the subclass $\mathcal{AT}_0$ if, in addition, it satisfies $\text{End}(G) = \text{Inn}(G)$ and has trivial abelianization. Concrete examples can be obtained by taking the free factors $G_1$ and $G_2$ to be property (T) wreath-like product groups as constructed in \cite{CIOS4}, as shown in Theorem \ref{afpwithnoend}. In particular, by specializing the preceding result to groups in the subclass $\mathcal{AT}_0$, we obtain the first examples of reduced $C^*$-algebras of nonamenable groups whose automorphism groups consist solely of weakly inner automorphisms.

\begin{cor}\label{AT0inner} Let $G \in \mathcal{AT}_0$. For every $\theta \in \mathrm{Aut}(C_r^*(G))$, there exists a unitary $w \in \mathrm{L}(G)$ such that $\theta = \mathrm{Ad}(w)$.
\end{cor}

\noindent For certain amalgams in $\mathcal{AT}$ we also obtain a nonembeddability result in the same spirit as the recent results from \cite{CIOS4} in the von Neumann algebraic setting. Namely consider the continuum family $\mathcal C_1$ consisting of all property (T) wreath-like groups $G_i\subset \mathcal{WR}(A_i, B_i \ca I_i)$ satisfying the conditions (i)-(iv) in Definition \ref{def:AT}, and such that for every $i\neq j$ we have that $G_i \not\hookrightarrow G_j$. These groups have been constructed in \cite[Corollary 1.7]{CIOS1} and \cite[Corollary 1.5]{CIOS4}. Consider the subclass $\mathcal{AT}_1\subset \mathcal{AT}$ consisting off all amalgams $G= G_1\ast_C G_2$ satisfying condition (v) for two distinct group $G_1, G_2 \in \mathcal C_1$ which contain a copy of a malnormal free group $C$. Notice that $\mathcal {AT}_1$ contains a continuum of elements as well.

\begin{cor}\label{nonembeddingAt1} For any two distinct groups $H_1, H_2\in \mathcal {AT}_1$ we have that $C^*_r(H_1)\not \hookrightarrow C^*_r(H_2)$.
\end{cor}

\paragraph{Outline of the paper.} Apart from the introduction, the paper has four more sections. In Section \ref{Sec:WR}, we construct examples in the classes $\mathcal{AT}$, $\mathcal{AT}_0$, and $\mathcal{AT}_1$. Section \ref{SEC:vNalgprel} collects von Neumann algebra preliminaries, which are used in Section \ref{SEC:superresults} to prove Theorem \ref{superr}. In Section \ref{SEC:product}, we establish a general product superrigidity result, based on \cite{CU18,CdSS15,CdSS17}, and prove Theorem \ref{superr2}. This section also contains the proofs of our embedding results for C$^*$-algebras, including Theorem \ref{ATembeddings} and Corollaries \ref{AT0inner} and \ref{nonembeddingAt1}.

\paragraph{Acknowledgments.}
The first author was supported by NSF Grants FRG-DMS-1854194 and DMS-2154637, the Graduate College Post-Comprehensive Research Fellowship, and the Erwin and Peggy Kleinfeld Scholar Fellowship. The second author was supported by NSF Grants DMS-2154637 and DMS-2452247. The third author was supported by FWO research project G016325N of the Research Foundation Flanders.

\section{Preliminaries on wreath-like product groups}\label{Sec:WR}

We recall the definition of the groups constructed in \cite{CIOS1}.

\begin{defn}\label{wlp} 
Let $A$, $B$ be arbitrary groups, $I$ an abstract set, $B\curvearrowright I$ a (left) action of $B$ on $I$.
We say that a group $W$ is a \emph{wreath-like product} of groups $A$ and $B$ corresponding to the action $B\curvearrowright I$ if $W$
is an extension of the form

\begin{equation}\label{ext}
    1\longrightarrow \bigoplus_{i\in I}A_i \longrightarrow  W \stackrel{\pi}\longrightarrow B\longrightarrow 1,
\end{equation}

\noindent where the following conditions are satisfied:

\begin{enumerate}
    \item[(a)] $A_i\cong A$, for each $i\in I$, and
    \item[(b)] the action of $W$ on $A^{(I)}=\bigoplus_{i\in I}A_i$ by conjugation satisfies the rule $wA_iw^{-1} = A_{\pi(w)\cdot i}$, for all $i\in I$.
\end{enumerate}


\noindent If the action $B\curvearrowright I$ is regular (i.e., free and transitive), we say that $W$ is a \emph{regular wreath-like product} of $A$ and $B$. The set of all wreath-like  products of groups $A$ and $B$ corresponding to an action $B\curvearrowright I$ (respectively, all regular wreath-like products) is denoted by $\WR(A, B\curvearrowright I)$ (respectively, $\WR(A,B)$).
\end{defn}

The notion of a wreath-like product generalizes the ordinary (restricted) wreath product of groups. Indeed, for any groups $A$ and $B$, we have $A {\text wr}  B\in \WR(A,B)$. Conversely, it is not difficult to show that $W\cong A {\text wr}  B$ whenever the extension \eqref{ext} splits. 

To enhance the analysis of II$_1$ factors arising from wreath-like product groups $W$, it is useful to embed specific groups into the quotient group $B$ from \eqref{ext} in such a way that these subgroups also naturally appear inside the wreath-like product $W$ itself. Motivated by this perspective, a \emph{peripheral} version of wreath-like product groups was introduced in \cite{CIOS5}. Let $A$ and $B$ be groups, and let $C < B$ be a subgroup. Throughout this paper, we denote by $\W\R(A, B, C < B)$ the category of all wreath-like groups $W \in \W\R(A, B)$ satisfying the following splitting condition:

\begin{enumerate}
    \item [(c)]\label{split} The subgroup $\pi^{-1}(C)<W$ is a split generalized wreath product group $\pi^{-1}(C)=A \wr_B   C$; in other words the natural $2$-cocycle associated with the extension \eqref{ext} is trivial on $C\times C$.
\end{enumerate}

\noindent Notice that $\W\R(A, B, C<B)\subseteq \W\R(A, B)$, and when $C=1$ the two classes coincide. 

In \cite{CIOS5}, the authors showed there is an abundance of property (T) such wreath-like products satisfying this splitting condition on prescribed subgroups of the quotient group. 
	
\begin{thm}[{\cite[Corollary 3.4]{CIOS5}}]\label{wreathlikea}For every finitely generated groups $A$ and $C$, there exist a group $B$, containing $C$, and a wreath-like product group $W\in \W\R(A, B, C<B)$ satisfying: \begin{enumerate}
    \item[(1)] $W$ has property (T);
    \item[(2)] $B$ is hyperbolic relative to $C$; and
    \item[(3)] When $A$ and $C$ are torsion free, then $B$, and hence $W$, can be chosen torsion free as well.
\end{enumerate}
\end{thm}

\vspace{1mm}


\begin{cor}\label{constructperipheral} For every finitely generated groups $A$ and $C$, there exists a a property (T) wreath-like product group $W\in \mathcal{WR}(A,B,C\leqslant B)$ such that $B$ is hyperbolic relative to $C$. In addition, if $A$ and $C$ are torsion free, then $W$ can be chosen torsion free as well. Moreover, in this case, if one assumes that the centralizer of every nontrivial element of $C$ is amenable then $B$ can be chosen to satisfy the same property. 
\end{cor}   

\begin{proof} Using Theorem \ref{wreathlikea} we only need to prove the moreover part. Assume that $B$ is torsion free and hyperbolic relative to $C<B$. Fix an element $b\in B$. If $b$ cannot be conjugated in $C$ then is hyperbolic and, by \cite[Lemma 6.5]{DGO11}, its centralizer $C_B(b)$ is virtually cyclic, and hence amenable. Now assume there is $k\in B$ such that $t:=kbk^{-1}\in C$. Since $t$ has infinite order and $C$ is malnormal in $B$ the centralizer satisfies $C_B(t)= C_C(t)$. By the additional assumption on $C$ we get that $C_B(t)$ is amenable, and so is $C_B(b)= k^{-1}C_B(t)k$. 
\end{proof}

\noindent Given an inclusion of groups $C < G$, we say that $C$ is \emph{almost malnormal} if for every $g \in G \setminus C$ the intersection $C \cap g C g^{-1}$ is finite; or equivalently, if for every $g,h\in G\setminus C$ the intersection $C\cap gCh$ is finite.

\begin{lem}[{\cite{CIOS5}}]\label{malnormal3}
If $C < B$ is almost malnormal, then for any $W \in \W\R(A, B, C < B)$, the subgroup $C < W$ is also almost malnormal.
\end{lem}

\noindent Next we present a proof of the previous corollary in the case when $W$ is torsion free and $C$ is any free group $C=\mathbb F_n$ which does not appeal to the results from \cite{CIOS5}.

\begin{thm}\label{existence} For every $n, k\in \mathbb N$, there is a property (T) group $W\in \mathcal {WR}(A, B, C<B )$ such that \begin{enumerate}
    \item[(i)] $A \cong \mathbb Z^k$ and $C\cong \mathbb F_n$;
    \item[(ii)] $B$ is icc, torsion free, and hyperbolic relative to $C<B$;
    \item[(iii)] $C$ is malnormal in $W$; and
    \item[(iv)] for every $1\neq b\in B$ the centralizer $C_B(b)$ is virtually cyclic.
\end{enumerate}
\end{thm}

\begin{proof} Fix $n\in \mathbb N$ and let $G$ be a torsion free property (T) group that is hyperbolic relative to $C<G$ where $C\cong \mathbb F_n$ (see for instance \cite[Theorem 1]{BO}). Using \cite[Proposition 4.21]{CIOS3} and \cite[Corollary 4.18]{CIOS3} there is a quotient $U$ of $G$ such that $U\in \mathcal WR(\mathbb F_k, B)$, where $B$ is torsion free and hyperbolic relative to a copy of $C<B$. Note that, since $C$ is free it follows from \cite[Corollary 1.14]{DS05} that $B$ is hyperbolic. Since $B$ is torsion free, it is icc. Moreover, by \cite[Corollary 2.6]{CIOS4}, $B$ also satisfies condition (iv) on centralizers.

Write $A=\mathbb Z_k = \mathbb F_k/N$ for some normal subgroup $N\lhd \mathbb F_k$. Then applying \cite[Lemma 4.3]{CIOS3} for $U$ and $N$ we further obtain a property (T) quotient $W \in \mathcal{WR}(A, B)$. Next, let $\pi: W \ra B$ be the canonical homomorphism. Since $C$ is free group, it follows that the subgroup $\pi^{-1}(C)$ is an extension $1\ra A^{(B)}\hookrightarrow \pi^{-1}(C)\overset{\pi}{\twoheadrightarrow} C \ra 1$ that splits, which means $W \in \mathcal{WR}(A, B, C<B)$. The fact that $C$ is almost malnormal in $W$ now follows directly from Lemma \ref{malnormal3}. \end{proof}

\begin{lem}\label{malnormal4}
Let $G_1$ and $G_2$ be groups with a common subgroup $C < G_1, G_2$. If $C < G_i$ is almost malnormal for each $1 \leq i \leq 2$, then $C < G_1 \ast_C G_2$ is also almost malnormal.
\end{lem}

\begin{proof} Let $g\in G\setminus C$ with $g=g_1g_2...g_n$ its reduced form for $n\geq 1$ an integer, $j(k)\in\{1,2\}$ and $g_k\in G_{j(k)}\setminus C$, for every $1\leq k\leq n$, and $j(1)\neq j(2)\neq \cdots\neq j(k)$. Notice that for $x\in C\cap gCg^{-1}$, $g^{-1}xg=g_n^{-1}\cdots g_1^{-1}xg_1\cdots g_n\in C$, which forces $g_1^{-1}xg_1\in C$. Therefore, $C\cap gCg^{-1}\subset C\cap g_1Cg_1^{-1}$, which is finite since $g_1\in G_{j(1)}\setminus C$ and $C$ is malnormal in $G_{j(1)}$.
\end{proof}

\noindent We continue with a result illustrating that Theorem \ref{existence} can be effectively used to produce numerous examples of amalgamated free products of wreath-like groups satisfying the assumptions of our main rigidity results in Theorem \ref{ATrigid}.

\begin{thm} 
There exist infinitely many torsion free,  amalgamated free product groups $G = G_1 \ast_C G_2$ satisfying the following properties: 
\begin{enumerate}
    \item[(i)] Each $G_i \in \mathcal{WR}(A_i, B_i, C < B_i)$ is a property (T) wreath-like product group, where $A_i$ is a free abelian group of any given finite rank, and $B_i$ is a specific torsion free, icc hyperbolic group satisfying that, for every $1\neq b\in B_i$, its centralizer $C_{B_i}(b)$ is virtually cyclic.
    \item[(ii)] $C$ is isomorphic to a free group of any given finite rank and it is malnormal in $G$. In particular, when $C$ is icc, its virtual centralizer in $G$ satisfies ${\rm vC}_G(C)=\{1\}$.
    \item [(iii)] For every $h\neq 1$, its centralizer $C_G(h)$ is amenable. 
\end{enumerate}
\end{thm}

\begin{proof}
Fix $n,k \in \mathbb{N}$ with $n \geq 2$, and set $C = \mathbb{F}_n$ and $A_i = \mathbb{Z}^k$. By Theorem~\ref{existence}, for each $1 \leq i \leq 2$, there exists a torsion-free property~(T) wreath-like product group $G_i \in \mathcal{WR}(A_i, B_i, C < B_i)$ satisfying all conditions listed in (i). Since $C$ is malnormal in each $G_i$, Lemma \ref{malnormal4} implies that $C$ is malnormal in the amalgamated free product $G = G_1 \ast_C G_2$. In particular, this implies that ${\rm vC}_G(C) \leq  {\rm vC}_C(C)$. Since $C$ is icc, we have ${\rm vC}_C(C) = \{1\}$, which further yields ${\rm vC}_G(C) = \{1\}$. Notice that every nontrivial element in $G$ has amenable stabilizer by \cite[Theorem 1]{KS07}. Moreover, as each $G_i$ is torsion free, it follows that $G$ is also torsion free. 

Finally, one can see that, varying $n\in\mathbb N$ or $k\in \mathbb N$, there are infinitely many such groups.
\end{proof}

\noindent Recall that in \cite{CIOS4}, property (T) wreath-like product groups were constructed with the additional properties that they have trivial abelianization and that all their injective endomorphisms are automatically inner automorphisms. In fact, \cite[Theorem 2.10]{CIOS4} shows that there exist infinitely many such groups. We denote the category of these groups by $\mathcal{C}_0$. In the following theorem, we show that any amalgamated free product of two groups in $\mathcal{C}_0$ over an free malnormal subgroup as in \cite{CIOS4} belongs to the class $\mathcal{AT}_0$. We use the notions of intertwining $\prec$ in the sense of Popa, recalled in Theorem \ref{corner}, and amalgamated free product of von Neumann algebras, recalled in Section \ref{SSec:amalgamated}. We include the theorem in here since it is a group-theoretic result and fits naturally with the theme of this section.

\begin{thm}\label{afpwithnoend} If $G_1,G_2 \in \mathcal C_0$ with $G_1\ncong G_2$ then for every $G= G_1\ast_C G_2\in \mathcal {AT}$ we have $G \in \mathcal {AT}_0$.
\end{thm}

\begin{proof} To get our conclusion we only need to show that an injective endomorphism of $G$ is an inner automorphism and that $G$ has trivial abelianization.

To see the second assertion, fix $A$ an abelian group and let $\delta: G \ra A$ be any group homomorphism. Since the restriction $\delta_{|G_i}: G_i \ra A$ is also a group homomorphism and $G_i$ has trivial abelianization we have $G_i =\ker(\delta_{|G_i})\leqslant  \ker (\delta)$ for all $i=1,2$. This further implies that $G = G_1\ast_ CG_2=G_1\vee G_2\subset\ker(\delta) $. Therefore, $G$ has trivial abelianization. 

Now let $\delta: G \ra G$ be any injective group endomorphism. Since $\delta(G_i)< G= G_1 \ast_C G_2$ is a property (T) subgroup, \cite[Theorem 5.1]{IPP05} gives that ${\rm L}(\delta(G_i))\prec_{{\rm L}(G)} {\rm L}(G_{j_i})$, for some $j_i\in\{1,2\}$. Using \cite[Lemma 2.2]{CI17}, there is $g_i\in G$ such that  $[\delta(G_i): \delta(G_i)\cap g_iG_{j_i}g_i^{-1}]<\infty$. This means that $[g^{-1}_i\delta(G_i)g_i: g^{-1}_i\delta(G_i)g_i\cap G_{j_i}]<\infty$. Since $G_i$ has property (T), so does $g_i^{-1}\delta(G_i)g_i\cap G_{j_i}$, and thus, ${\rm L}(g_i^{-1}\delta(G_i)g_i\cap G_{j_i})\nprec_{{\rm L}(G_{j_i})} {\rm L}(C)$. Thus, by \cite[Theorem 1.1]{IPP05} we have that ${\rm L}(g_i^{-1}\delta(G_i) g_i)\subseteq {\rm L}(G_{j_i})$, and hence 

\begin{equation}\label{conjfact}
    g^{-1}_i\delta(G_i) g_i \leqslant  G_{j_i}, \text{ for all }i=1,2.
\end{equation}

\noindent Next we show that $\{j_1,j_2\}$ is a permutation of $\{1,2\}$. Assume by contradiction that $j_1=j_2=1$. From relation \eqref{conjfact}, ${\rm Ad}(g_1)\circ \delta: G_1 \ra G_1$ is an injective endomorphism. Form the assumptions on $G_1$, it follows that there is $h\in G$ such that the restriction $\delta|_{G_1}$ is an isomorphism $\delta: G_1 \ra h G_1  h^{-1}$ given by $\delta(g)= hgh^{-1}$, for all $g\in G_1$. From \eqref{conjfact} we also have that $g^{-1}_2\delta(G_2)g_2\leqslant G_1$. Altogether, these relations show that $g_2^{-1}hCh^{-1}g_2=g_2^{-1}\delta|_{G_1}(C)g_2=g_2^{-1}\delta|_{G_2}(C)g_2\leqslant G_1$, and hence $g_2^{-1}h\in G_1$. This implies $\delta(G_2)\leqslant h G_1 h^{-1}$. Since $G_2$ has property (T) and $C$ is Haagerup, there exists $t\in G_1\setminus C$ and $k_2 \in G_2\setminus C$ such that $\delta(k_2)= ht h^{-1}$. On the other hand, since $\delta|_{G_1}$ is surjective onto $hG_1h^{-1}$, there is $k_1 \in G_1$ such that $\delta(k_1)=hth^{-1}$. Altogether, these show that $\delta(k_1)=\delta(k_2)$, and hence, $k_1^{-1}k_2\in \ker (\delta)$. On the other hand, by construction $k_1^{-1}k_2\neq 1$, contradicting that $\delta$ is injective.

Now assume by contradiction that $j_1=2$ and $j_2=1$. Letting $\delta_i ={\rm Ad}(g_i^{-1})\circ \delta$, relations \eqref{conjfact} show that $\delta_1 : G_1 \ra G_2$ and $\delta_2: G_2 \ra G_1$ are injective group homomorphisms. This means $\delta_2\circ \delta_1: G_1\ra G_1$ and $\delta_1\circ \delta_2: G_2\ra G_2$ are injective endomorphisms and from assumptions they must be inner automorphisms. In particular, this shows that $G_1 \cong G_2$, a contradiction. 

Hence, we must have that $j_i=i$. In this case relations \eqref{conjfact} and the assumptions imply that there exist $g_1,g_2\in G$ such that the restrictions satisfy $\delta: G_1 \ra g_1 G_1 g_1^{-1}$ and $\delta: G_2 \ra g_2G_2 g_2^{-1}$ and are given by $\delta(g)= {\rm Ad}(g_i)(g)$ for every $g\in G_i$. In particular, $g_1 cg_1^{-1}= g_2 cg_2^{-1}$ for all $c\in C$. Since the centralizer of $C$ in $G$ is trivial we get that $g_1=g_2$. This yields that $\delta= {\rm Ad} (g_1)$, as desired. 
\end{proof}

\noindent Consider the continuum family $\mathcal C_1$ introduced in \cite[Corollary 1.5]{CIOS4} which consists of all wreath-like product groups $G\in \mathcal WR(A, B\ca I)$ satisfying conditions (i)-(iv) from Definition \ref{def:AT} with the additional property that $G_1 \not\hookrightarrow G_2$ for any distinct groups $G_1, G_2 \in \mathcal C_1$. Consider the subclass $\mathcal{AT}_1\subset \mathcal{AT}$ consisting of all amalgamated free product groups $G= G_1\ast_C G_2$ for distinct groups $G_1, G_2 \in \mathcal C_1$ containing a copy of a malnormal free group $C$. 

\begin{thm}\label{AT1nonembedding} We have $|\mathcal{AT}| = 2^{\aleph_0}$, and for every pair of distinct groups $H_1, H_2 \in \mathcal{AT}_1$, one has $H_1 \not\hookrightarrow H_2$.
\end{thm}

\begin{proof}
For $i = 1,2$, let $H_i = G^i_1 \ast_{C_i} G^i_2$, and assume that $H_1 \ncong H_2$. By the choice of these groups, this implies that either $G^1_1 \ncong G^2_1$ or $G^1_2 \ncong G^2_2$. Suppose, toward a contradiction, that there exists a group monomorphism $\theta \colon H_1 \to H_2$. The same arguments as those used at the beginning of the proof of Theorem~\ref{afpwithnoend} show the existence of an element $g_j \in H_2$ with $g_j \theta(G^1_j) g_j^{-1} \leqslant G^2_{k_j}$, for some $k_j \in \{1,2\}$. Since these groups belong to the class $\mathcal{C}_1$, it follows that $k_j = j$ for both $j = 1,2$. Moreover, we obtain that $G^1_1 \cong G^2_1$ and $G^1_2 \cong G^2_2$, which contradicts our initial assumption.
\end{proof}


\section{Preliminaries on von Neumann algebras}\label{SEC:vNalgprel}

We start by recalling some terminology and constructions involving tracial von Neumann algebras along with several results, mainly from \cite{CIOS1}.

A {\it tracial von Neumann algebra} is a pair $(\M,\tau)$ consisting of a von Neumann algebra $\M$ and a normal faithful tracial state $\tau\colon\M\rightarrow\mathbb C$. For $x\in \M$, we denote by $\|x\|$ the operator norm of $x$ and by $\|x\|_2=\tau(x^*x)^{1/2}$ its $2$-norm. We denote by $L^2(\M)$ the Hilbert space obtained as the closure of $\M$ with respect to the $2$-norm, by $\sU(\M)$ the group of {\it unitaries}  of $\M$, and by $(\M)_1=\{x\in \M :  \|x\|\leq 1\}$ the {\it unit ball} of $\M$. We will always assume that $\M$ is {\it separable} or, equivalently, that $L^2(\M)$ is a separable Hilbert space. We denote by ${Aut}(\M)$ the group of $\tau$-preserving automorphisms of $\M$.  For $u\in\sU(\M)$, the {\it inner} automorphism $\text{Ad}(u)$ of $\M$ is given by $\text{Ad}(u)(x)=uxu^*$. By von Neumann's bicommutant theorem, for any set $X\subset \M$ closed under adjoint, $X''\subset \M$ is the smallest von Neumann subalgebra which contains $X$.



Let $\Q\subset \M$ be a von Neumann subalgebra, which we always assume to be unital. We denote by $\Q'\cap \M=\{\text{$x\in \M\mid xy=yx$, for all $y\in \Q$}\}$ the {\it relative commutant} of $\Q$ in $\M$, and by $\sN_\M(\Q)=\{u\in\sU(\M)  :  u\Q u^*=\Q\}$ the {\it normalizer} of $\Q$ in $\M$. The {\it center} of $\M$ is given by $\sZ(\M)=\M'\cap \M$. We say that $\M$ is a {\it factor} if $\sZ(\M)=\mathbb C1$ and that $\Q\subset \M$ is an {\it irreducible subfactor} if $\Q'\cap \M=\mathbb C1$. We say that $\Q$ is {\it regular} in $\M$ if $\sN_\M(\Q)''=\M$. If $\Q\subset \M$ is regular and maximal abelian, we call it a {\it Cartan subalgebra}.

{\it Jones' basic construction} $\langle \M,e_\Q\rangle$ is defined as the von Neumann subalgebra of $\mathbb B(L^2(\M))$ generated by $\M$ and the orthogonal projection $e_\Q$ from $L^2(\M)$ onto $L^2(\Q)$. The basic construction $\langle \M,e_\Q\rangle$ has a faithful semi-finite trace given by $\text{Tr}(xe_\Q y)=\tau(xy)$, for every $x,y\in \M$. We denote by $L^2(\langle \M,e_\Q\rangle)$ the associated Hilbert space and endow it with the natural $\M$-bimodule structure. We also denote by $E_\Q\colon\M\rightarrow \Q$ the unique $\tau$-preserving {\it conditional expectation} onto $\Q$.

The tracial von Neumann algebra $(\M,\tau)$  is called {\it amenable} if there exists a sequence $\xi_n\in L^2(\M)\otimes L^2(\M)$ such that $\langle x\xi_n,\xi_n\rangle\rightarrow\tau(x)$ and $\|x\xi_n-\xi_nx\|_2\rightarrow 0$, for every $x\in \M$. Let $\P\subset p\M p$ be a von Neumann subalgebra. Following Ozawa and Popa \cite[Section 2.2]{OP07} we say that $\P$ is {\it amenable relative to $\Q$ inside $\M$} if there exists a sequence $\xi_n\in L^2(\langle \M,e_\Q\rangle)$ such that $\langle x\xi_n,\xi_n\rangle\rightarrow\tau(x)$, for every $x\in p\M p$, and $\|y\xi_n-\xi_ny\|_2\rightarrow 0$, for every $y\in \P$.
We say that $\P$ is {\it strongly nonamenable relative to $\Q$ inside $\M$} if there exist no nonzero projection $p'\in \P'\cap p\M p$ such that $\P p'$ is amenable relative to $\Q$ inside $\M$.




\subsection {Intertwining-by-bimodules}

We recall from  \cite [Theorem 2.1, Corollary 2.3]{Po03} Popa's {\it intertwining-by-bimodules} theory.

\begin{thm}[\cite{Po03}]\label{corner} Let $(\M,\tau)$ be a tracial von Neumann algebra and $\P\subset p\M p, \Q\subset q\M q $ be von Neumann subalgebras. Then the following conditions are equivalent:

\begin{enumerate}
    \item[(a)] There exist projections $p_0\in \P, q_0\in \Q$, a $*$-homomorphism $\theta\colon p_0\P p_0\rightarrow q_0\Q q_0$  and a nonzero partial isometry $v\in q_0\M p_0$ such that $\theta(x)v=vx$, for all $x\in p_0\P p_0$.
    \item[(b)] There is no sequence $u_n\in\sU(\P)$ satisfying $\|E_\Q(x^*u_ny)\|_2\rightarrow 0$, for all $x,y\in p\M$.
\end{enumerate}

\noindent If these conditions hold true, we write $\P\prec_{\M}\Q$ and say that {\it a corner of $\P$ embeds into $\Q$ inside $\M$.}
If $\P p'\prec_{\M}\Q$, for any nonzero projection $p'\in \P'\cap p\M p$, we write $\P\prec^{\text s}_{\M}\Q$.
\end{thm}

\noindent The following two results are taken from \cite{CIOS1} and will be used in the proof of one of our main rigidity results. The first is about the structure of normalizers in crossed products arising from actions of hyperbolic groups based on \cite{PV12} and proven in \cite[Theorem 3.10]{CIOS1}, while the second is a solidity result established in \cite[Corollary 4.7]{CIOS1}.

\begin{thm}\label{relativeT} Let $G,H$ be countable groups, with $H$ a non-elementary subgroup of a hyperbolic group. Let $\delta:G\to H$ be a homomorphism and $G\curvearrowright(\mathcal{Q},\tau)$ a trace preserving action on a tracial von Neumann algebra $(\mathcal{Q},\tau)$. Let $\M=\mathcal{Q}\rtimes G$ and $\mathcal{P}\subset p\M p$ be an amenable von Neumann subalgebra. Assume there is a von Neumann subalgebra $\mathcal{R}\subset \mathscr{N}_{p\M p}(\mathcal{P})''$ with property (T) such that $\mathcal{R}\not\prec_{\M}\mathcal{Q}\rtimes\ker(\delta)$. Then $\mathcal{P}\prec_{\M}^s\mathcal{Q}\rtimes\ker(\delta)$.
\end{thm}

\begin{thm}\label{solidity} For $i=1,2$, let $G_i\in\mathcal{WR}(A_i,B_i\curvearrowright I_i)$, where $A_i$ is an abelian group and $B_i\curvearrowright I_i$ an action such that $\text{Stab}_B(j)$ is amenable for every $j\in I_i$ and $\{j\in I:g\cdot j\neq j\}$ is infinite for every $g\in B\setminus\{1\}$. Let $\mathcal{Q}\subset p(\text{L}(G_1)\bar{\otimes}\text{L}(G_2)\bar{\otimes}\mathbb M_n(\mathbb C))p$ be a von Neumann subalgebra such that $\mathcal Q\prec^s\text{L}(A_1)\bar{\otimes}\text{L}(A_2)\bar{\otimes}\mathbb M_n(\mathbb C)$ and $\mathcal{Q}\not\prec \text{L}(A_1)\bar{\otimes}\mathbb M_n(\mathbb C),\text{L}(A_2)\overline{\otimes}\mathbb M_n(\mathbb C)$. Then $\Q'\cap p\M p$ is amenable.
\end{thm}

\subsection{Amalgamated free product of von Neumann algebras}\label{SSec:amalgamated}

In this section, we briefly recall the construction of the amalgamated free product for tracial von Neumann algebras. Let $(\M_1,\tau_1)$ and $(\M_2,\tau_2)$ be tracial von Neumann algebras with common von Neumann subalgebra $B$ such that $\tau_1|_B=\tau_2|_B$. Let $\mathbb{E}_{i}:\M_i\to B$ be the trace-preserving conditional expectation onto $B$, for $i=1,2$. The amalgamated free product $(\M,\mathbb{E}_B)=(\M_1,\mathbb{E}_1)\ast_B(\M_2,\mathbb{E}_2)$ is a pair consisting of the von Neumann algebra $\M$ generated by $\M_1$ and $\M_2$ and a trace-preserving conditional expectation $\mathbb{E}_B:\M\to B$ (with respect to its canonical trace $\tau$) such that $\M_1$, $\M_2$ are freely independent with respect to $\mathbb{E}_B$; that is, $\mathbb{E}_B(x_1\cdots x_n)=0$ whenever $x_j\in \M_{i_j}\ominus B$ and $i_1\neq\cdots\neq i_n$. We refer to the product $x_1\cdots x_n$, where $x_j\in \M_{i_j}\ominus B$ and $i_1\neq\hdots\neq i_n$, as a reduced word. The linear span of $B$ and all reduced words form a strongly dense $*$-subalgebra of $\M$. We call the resulting $\M$ the \textit{amalgamated free product von Neumann algebra} of $(\M_1,\tau_1)$ and $(\M_2,\tau_2)$.

In this paper, we primarily work with von Neumann algebras associated to amalgamated free product groups $G=G_1\ast_CG_2$. In this setting, one has $\text{L}(G)=\text{L}(G_1)\ast_{\text{L}(C)}\text{L}(G_2)$, as defined above. The only result we require for general amalgamated free product von Neumann algebras is the statement below, which follows as a consequence of \cite{Io12} and \cite{Va13}.

\begin{thm}\label{commutlocinamalgams} Let $\mathcal M=\mathcal M_1\ast_{\mathcal{B}}\mathcal M_2$ be an amalgamated free product of tracial von Neumann algebra and let $p\in \mathcal M$ be a nonzero projection. Let $\mathcal{C}_1,\mathcal{C}_2\subset p\mathcal M p$ be any two non-amenable commuting von Neumann subalgebras. Then one of the following must hold:

\begin{enumerate}
    \item $\mathcal C_i \prec_\mathcal{M} \mathcal{B}$ for some $1\leq i\leq 2$;
    \item $\mathcal{C}_1\vee \mathcal{C}_2 \prec_\mathcal{M} \mathcal{M}_i $ for some $1\leq i\leq 2$;
    \item $\mathcal{C}_1\vee\mathcal{C}_2$ is amenable relative to $\mathcal{B}$ inside $\mathcal{M}$.
\end{enumerate}
\end{thm}

\begin{proof} Assume $\mathcal{C}_2\not\prec_{\mathcal{M}}\mathcal{B}$, and fix $\mathcal{A}\subset \mathcal{C}_1$ an amenable subalgebra. By \cite[Theorem A]{Va13} we either have:

\begin{enumerate}
    \item[(i)] $\mathcal{A}\prec_{\mathcal{M}}\mathcal{B}$,
    \item[(ii)] $\mathscr N_{p\mathcal M p}(\mathcal{A})''\prec_{\mathcal{M}}\mathcal{M}_i$, for some $i$, or
    \item[(iii)] $\mathscr N_{p\mathcal M p}(\mathcal{A})''$ is amenable relative to $\mathcal{B}$ inside $\mathcal{M}$.
\end{enumerate}

\noindent If (i) holds for all amenable subalgebras $\mathcal{A}\subset\mathcal{C}_1$, $\mathcal{C}_1\prec_{\mathcal{M}}\mathcal{B}$ by \cite[Corollary F.14]{BO08}, obtaining (1). Suppose (i) doesn't hold for some amenable $\mathcal{A}\subset \mathcal{C}_1$, and (ii) holds. In that case, since $\mathcal{C}_2\prec\mathcal{M}_i$ and $\mathcal{C}_2\not\prec\mathcal{B}$, using \cite[Theorem 1.1]{IPP05} we obtain $\mathscr N_{p\mathcal{M}p}(\mathcal{C}_2)''\prec_{\mathcal{M}}\mathcal{M}_i$. This means $\mathcal{C}_1\vee\mathcal{C}_2\prec_{\mathcal{M}}\mathcal{M}_i$, obtaining (2). Now, suppose neither (i) nor (ii) hold for some amenable subalgebra $\mathcal{A}\subset\mathcal{C}_1$. From (iii) we have $\mathcal{C}_2$ is amenable relative to $\mathcal{B}$ inside $\mathcal{M}$. Another application of \cite[Theorem A]{Va13}, gives that one of the following must hold:

\begin{enumerate}
    \item[(iv)] $\mathcal{C}_2\prec\mathcal{B}$,
    \item[(v)] $\mathscr N_{p\mathcal M p}(\mathcal{C}_2)''\prec_{\mathcal{M}}\mathcal{M}_i$, for some $i$, or
    \item[(vi)] $\mathscr N_{p\mathcal M p}(\mathcal{C}_2)''$ is amenable relative to $\mathcal{B}$ inside $\mathcal{M}$.
\end{enumerate}

\noindent By assumption, (iv) cannot hold, and if (v) holds, \cite[Theorem 1.1]{IPP05} gives $\mathscr N_{p\mathcal{M}p}(\mathcal{A})''\prec\mathcal{M}_i$, a contradiction. Thus, we must have that (vi) holds, and therefore, $\mathcal{C}_1\vee\mathcal{C}_2$ is amenable relative to $\mathcal{B}$ inside $\mathcal{M}$, obtaining (3).
\end{proof}

\subsection{Cartan subalgebras and equivalence relations}

For further use, we recall two conjugacy results for Cartan subalgebras from \cite{Io10}, which are proved in \cite[Lemmas 3.7 and 3.8]{CIOS1}.

\begin{lem}[\cite{Io10}]\label{conj1}
Let $\M$ be a II$_1$ factor, $\A\subset \M$ be a Cartan subalgebra and $\D\subset \M$ be an abelian von Neumann subalgebra. Let $\C=\D'\cap \M$ and assume that $\C\prec_{\M}^{s}\A$. Then there exists $u\in\sU(\M)$ such that $\D\subset u\A u^*\subset \C$.
\end{lem}

\begin{lem}[\cite{Io10}]\label{conj2}
Let $\M$ be a II$_1$ factor, $\A\subset \M$ a Cartan subalgebra, $\D\subset \M$ an abelian von Neumann subalgebra and let $\C=\D'\cap \M$. Assume that $\C\prec_{\M}^{\text{s}}\A$ and $\D\subset \A\subset \C$. Let $(\alpha_g)_{g\in G}$ be an action of a group $G$ on $\C$ such that $\alpha_g=\emph{Ad}(u_g)$, for some $u_g\in\sN_\M(\D)$, for every $g\in G$. Assume that the restriction of the action $(\alpha_g)_{g\in G}$  to $\D$ is free. Then there is an action $(\beta_g)_{g\in G}$ of $G$ on $\C$ such that

\begin{enumerate}
    \item\label{unu} for every $g\in G$ we have that $\beta_g=\alpha_g\circ\emph{Ad}(\omega_g)=\emph{Ad}(u_g\omega_g)$, for some $\omega_g\in\sU(\C)$, and
    \item\label{doi} $\A$ is $(\beta_g)_{g\in G}$-invariant and the restriction of $(\beta_g)_{g\in G}$ to $\A$ is free.
\end{enumerate}
\end{lem}

\subsection{Strong rigidity for orbit equivalence embeddings}\label{STRONG}

The following generalization of \cite[Theorem 0.5]{Po04} to a large class of actions built over was obtained in \cite[Theorem 4.1]{CIOS1}. If $\mathscr R$ is a countable p.m.p. equivalence relation on a standard probability space $(X,\mu)$, then the \textit{full group} $[\mathscr R]$ consists of all automorphisms $\theta$ on $(X,\mu)$ such that $(\theta(x),x)\in\mathscr R$, for almost every $x\in X$. If $G\curvearrowright (X,\mu)$ is a p.m.p. action, then $\mathscr R(G\curvearrowright X)=\{(x_1,x_2)\in X\times X:G\cdot x_1=G\cdot x_2\}$ is its \textit{orbit equivalence relation}.

\begin{thm}\label{SOE} Let $B$ be an icc group and $ B\curvearrowright^{\alpha}(X,\mu)=(Y^I,\nu^I)$ be a measure preserving action built over an action $B\curvearrowright I$, where $(Y,\nu)$ is a probability space. Let $\widetilde B= B\times\mathbb Z/n\mathbb Z$ and $(\widetilde X,\widetilde\mu)=(X\times\mathbb Z/n\mathbb Z,\mu\times c)$, where $n\in\mathbb N$ and $c$ is the counting measure of $\mathbb Z/n\mathbb Z$. Consider the action $\widetilde B\curvearrowright^{\widetilde\alpha} (\widetilde X,\widetilde\mu)$ given by $(g,a)\cdot (x,b)=(g\cdot x,a+b)$.

Let $D$ be a countable group with a normal subgroup $D_0$ such that the pair $(D,D_0)$ has the relative property (T).
Let $X_0\subset \widetilde X$ be a measurable set and $D\curvearrowright^{\beta}(X_0,\widetilde\mu_{|X_0})$ be a weakly mixing free measure preserving action such that $D\cdot x\subset\widetilde B\cdot x$, for almost every $x\in X_0$.
Assume that for every $i\in I$ and $g\in B\setminus\{1\}$, there is a sequence $(h_m)\subset D_0$ such that for every $s,t\in\widetilde B$ we have 

\begin{align*}
    \lim_{m\to\infty}&\widetilde\mu(\{x\in X_0\mid h_m\cdot x\in s(\text{Stab}_B(i)\times\mathbb Z/n\mathbb Z)t\cdot x\})= 0, \text{ and}\\
    \lim_{m\to\infty}&\widetilde\mu(\{x\in X_0\mid h_m\cdot x\in s(\text{C}_B(g)\times\mathbb Z/n\mathbb Z)t\cdot x\})= 0.
\end{align*}

\noindent Then there exist a subgroup $B_1<B$, a group isomorphism $\delta\colon D\rightarrow B_1$, and $\theta\in [\sR(\widetilde B\curvearrowright\widetilde X)]$ such that $\theta(X_0)=X\equiv X\times\{0\}$ and $\theta\circ\beta(h)=\alpha(\delta(h))\circ\theta$, for every $h\in D$.
\end{thm}

\noindent The following result due to Popa \cite{Po04} (see also \cite[Lemma 4.4]{CIOS1}) shows that the weakly mixing condition from the prior theorem is automatically satisfied after passing to an ergodic component of a finite index subgroup.

\begin{lem}[\cite{Po04}]\label{wmix} Assume the setting of Theorem \ref{SOE}. Then there exists a finite index subgroup $ S< D$ and a $\beta(S)$-invariant non-null measurable set $Y\subset X_0$ such that $\widetilde\mu(\beta(h)(Y)\cap Y)=0$, for every $h\in D\setminus S$, and the restriction of $\beta_{| S}$ to $Y$ is weakly mixing.
\end{lem}

\section{W\texorpdfstring{$^*$}{*}-superrigidity results for groups in class \texorpdfstring{$\mathcal{AT}$}{AT}}\label{SEC:superresults}

In this section we show the class $\mathcal{AT}$ defined in Definition \ref{def:AT} consists of W$^*$-superrigid groups (Theorem \ref{ATrigid}). This result will be used to prove that finite direct sums of such groups are also W$^*$-superrigid (Theorem \ref{superr}). Before the proof, we introduce some notation and record several useful facts.

Let $A$ be a nontrivial abelian group, $B$ a nontrivial icc subgroup of a hyperbolic group and $B\curvearrowright I$ an action with $\text{Stab}_B(i)$ is amenable, for every  $i\in I$. Let $G\in\mathcal W\mathcal R(A,B\curvearrowright I)$ be a  property (T) group. Denote $\M=\text{L}(G)$ and assume that $\M^t=\text{L}(H)$, for a countable group $H$ and $t>0$. 

Let $\Delta_0\colon \text{L}(H)\rightarrow \text{L}(H) \bar{\otimes} \text{L}(H)$ be the comultiplication given by $\Delta_0(v_h)=v_h\otimes v_h$, for every $h\in H$. Let $n$ be the smallest integer such that $n\geq t$.
Denote $\sM=\M \bar{\otimes} \M \bar{\otimes} \mathbb M_n(\mathbb C)$.
Then $\Delta_0$ can be amplified to a unital $*$-homomorphism $\Delta\colon \M\rightarrow p\sM p$, where $p\in\sM$ is a projection with $(\tau\otimes\tau\otimes\text{Tr})(p)=t$.

\vspace{1mm}

\begin{rem} Assume that $t\in\mathbb N$, so that $n=t$, $p=1$ and $\text{L}( H)=\M^n=\M \bar{\otimes} \mathbb M_n(\mathbb C)$. For further reference, we make explicit the construction of $\Delta$ in terms of $\Delta_0$. To this end, let $\psi\colon \sM \bar{\otimes} \mathbb M_n(\mathbb C)\rightarrow \M \bar{\otimes} \mathbb M_n(\mathbb C) \bar{\otimes} \M \bar{\otimes} \mathbb M_n(\mathbb C)=\text{L}( H) \bar{\otimes} \text{L}( H)$ be the $*$-isomorphism given by $\psi(a\otimes b\otimes c\otimes d)=a\otimes c\otimes b\otimes d$, for every $a,b\in \M$ and $c,d\in\mathbb M_n(\mathbb C)$. Let $U\in \M \bar{\otimes} \mathbb M_n(\mathbb C) \bar{\otimes} \M \bar{\otimes} \mathbb M_n(\mathbb C)$ be a unitary such that $\Delta_0(1_\M\otimes x)=U\psi(1_{\sM}\otimes x)U^*$, for every $x\in\mathbb M_n(\mathbb C)$. Then $\psi^{-1}\circ\text{Ad}(U^*)\circ\Delta_0\colon \M \bar{\otimes} \mathbb M_n(\mathbb C)\rightarrow\sM \bar{\otimes} \mathbb M_n(\mathbb C)$ is a unital $*$-homomorphism which leaves $1 \bar{\otimes} \mathbb M_n(\mathbb C)$ fixed, so it can be written as $\Delta\otimes\text{Id}$, where $\Delta\colon \M\rightarrow\sM$ is the desired unital $*$-homomorphism that amplifies $\Delta_0$. Thus, we conclude

\begin{equation}\label{Delta}
    \Delta_0=\text{Ad}(U)\circ\psi\circ(\Delta\otimes\text{Id}).
\end{equation}

\noindent In the proof of Theorem \ref{superr}, we will combine \eqref{Delta} with the symmetry and associativity properties of $\Delta_0$:  $\zeta\circ\Delta_0=\Delta_0$ and $(\Delta_0\otimes\text{Id})\circ\Delta_0=(\text{Id}\otimes\Delta_0)\circ\Delta_0$, where $\zeta$ is the flip automorphism of $\text{L}( H) \bar{\otimes} \text{L}( H)$ given by $\zeta(x\otimes y)=y\otimes x$, for every $x,y\in \text{L}( H)$.
\end{rem}

\vspace{1mm}

\begin{lem}[\cite{IPV10}]\label{comultiply} Let $\Delta\colon \M\rightarrow p\sM p$ be as defined above. Then the following hold:
\begin{enumerate}
    \item[(a)] $\Delta(\Q)\nprec_{\sM}\M \bar{\otimes}  1 \bar{\otimes} \mathbb M_n(\mathbb C)$ and $\Delta(\Q)\nprec_{\sM}1 \bar{\otimes}  \M \bar{\otimes} \mathbb M_n(\mathbb C)$, for any diffuse von Neumann subalgebra $\Q\subset \M$.
    \item[(b)] $\Delta(\M)\nprec_{\sM}\M \bar{\otimes} \emph{L}(G_0) \bar{\otimes} \mathbb M_n(\mathbb C)$ and $\Delta(\M)\nprec_{\sM}\emph{L}(G_0) \bar{\otimes} \M \bar{\otimes} \mathbb M_n(\mathbb C)$, for any infinite index subgroup $G_0<G$.
    \item[(c)] If $\mathcal H\subset \emph{L}^2(p\sM p)$ is a $\Delta(\M)$-sub-bimodule which is right finitely generated, then we have  $\mathcal H\subset \emph{L}^2(\Delta(\M))$.
    \item[(d)] Assume $p=1$. Also let $\mathcal P \subseteq \mathcal M$ be any von Neumann algebra such that for every $1\neq h$ we have $\mathcal P \nprec_{\mathcal M} {\rm L}(C_H(h))$. Then $\Delta (\mathcal P)'\cap \mathscr M= \Delta(\mathcal P' \cap \mathcal M)$ and $\mathscr{QN}_{\mathscr M}(\Delta(\mathcal P))''= \Delta(\mathscr{QN}_{\mathcal M}(\mathcal P)'')$.
\end{enumerate}
\end{lem}

\begin{proof} We will only prove part (d). Using \cite[Proposition 7.2(3)]{IPV10} we have that $\Delta_0(\mathcal P)' \cap (\mathscr M \bar\otimes \mathbb M_n(\mathbb C))= \Delta_0(\mathcal P'\cap (\mathcal M\bar \otimes \mathbb M_n(\mathbb C)))= \Delta_0 ((\mathcal P'\cap \mathcal M )\bar\otimes \mathbb M_n(\mathbb C))$. Using this in combination with relation \eqref{Delta} we get 

\begin{align}\label{diagonal control}
    U\psi(\Delta (\mathcal P'\cap\mathcal M ) \otimes \mathbb M_n(\mathbb C))U^*&=U(\psi(\Delta (\mathcal P) \otimes 1)' \cap (\mathscr M \bar\otimes \mathbb M_n(\mathbb C) ))U^*\\
    &= U\psi((\Delta(\mathcal P )'\cap \mathscr M)\bar \otimes \mathbb M_n(\mathbb C))U^*.\nonumber
\end{align}

\noindent Hence, $\Delta (\mathcal P'\cap\mathcal M ) \otimes \mathbb M_n(\mathbb C) =(\Delta(\mathcal P )'\cap \mathscr M)\bar \otimes \mathbb M_n(\mathbb C)$ and since $\Delta (\mathcal P'\cap\mathcal M ) \subseteq \Delta(\mathcal P )'\cap \mathscr M$ we must have $\Delta (\mathcal P'\cap\mathcal M ) = \Delta(\mathcal P )'\cap \mathscr M$, as desired. The other assertion follows similarly. 
\end{proof}

\noindent In the proof of Theorem \ref{superr} we will also need the fact that the set $\{i\in I\mid b\cdot i\not=i\}$ is infinite, for every $b\in B\setminus \{1\}$. This more generally holds if $B$ is acylindrically hyperbolic:

\begin{lem}\label{infinite}
Let $B$ be an \emph{icc} group acting on a set $I$. Then the following hold:
\begin{enumerate}
\item[(a)]
Assume that $B$ is acylindrically hyperbolic and $\emph{Stab}_B(i)$ is amenable, for every $i\in I$. Then,  for every non-trivial $b\in B$, the set $\{ i\in I\mid b \cdot i\ne i\}$ is infinite.
\item[(b)]
Assume that $B\cdot i$ is infinite, for every $\in I$. Let $A$ be a group. Then every $G\in\WR(A,B\curvearrowright I)$ is \emph{icc}. 
\end{enumerate}
\end{lem}

\noindent Before proceeding with the main proof, we establish the following result, which will be used repeatedly in the proof of Theorem \ref{super+torsion}.

\begin{prop}\label{singlesupport} Let $G$ and $K$ be groups and let $\mathcal T$ be a von Neumann algebra. Assume that $G \ca^{\sigma}\mathcal T$ is a trace preserving action and denote by $\mathcal M = \mathcal T\rtimes_\sigma G$. Let $\delta_1,\delta_2:K\to G$ be group homomorphisms and $\rho_1,\rho_2: K\ra \mathscr U(\mathcal T)$ a group representations. Assume that for every finite index $K_0 \leqslant K$ subgroup we have that $C_G(\delta_1(K_0))=\{1\}$. If $y\in \mathcal M$ satisfies

\begin{equation}\label{conj}
     y u_{\delta_1(k)}\rho_1(k)= u_{\delta_2(g)} \rho_2(k)y, \text{ for all } k\in K,
\end{equation}

\noindent then there exist $c\in \mathscr U(\mathcal T)$ and $l\in G$ with $y= c u_{l}$, $\delta_2 ={\rm Ad}(l)\circ \delta_1$ and $\rho_2(k)= \sigma_{\delta_2(k)^{-1}}(c) \sigma_l(\rho_1(k))c^*$. 
\end{prop}

\begin{proof} Consider the Fourier expansion $y=\sum_g y_gu_g$ and using relation \eqref{conj} we have \begin{equation}\label{condexpanded}\sum_g y_g \sigma_{g \delta_1(k)}(\rho_1(k)) u_{g\delta_1(k)}=y u_{\delta_1(k)}\rho_1(k)= u_{\delta_2(g)} \rho_2(k)y=\sum_g \sigma_{\delta_2(k)}(\rho_2(k) y_g) u_{\delta_2(k)g}. \end{equation}
Fix $\epsilon>0$ such that $F=\{k\in K:\|y_{k}\|_2\geq \epsilon\}$ is nonempty. Since $\rho_1,\rho_2$ are unitary representations, relation \eqref{condexpanded} implies that $F$ is finite and  $F\delta_1(g)=\delta_2(g)F$. Thus, 

\begin{equation}\label{orbit}
    \delta_1(g)^{-1}F^{-1}F\delta_1(g)=F^{-1}F, \text{ for every }g\in K.
\end{equation} 

\noindent Notice that for any $g\in F^{-1}F$ relation \eqref{orbit} implies that there is a finite index subgroup $K_0 \leqslant K$ with $h \in C_{G}(\delta_1(K_0))$. Since the latter is trivial we get $g=1$. However, $F^{-1}F=\{1\}$ entails that $F$ consists of a single element. As this holds for all $\epsilon>0$, we get that $y = c u_{l}$ for some $c\in \mathscr U(\mathcal T)$ and $l\in G$. Relation \eqref{conj} also implies that $\delta_2 ={\rm Ad}(l)\circ \delta_1$ and $\rho_2(k)= \sigma_{\delta_2(k)^{-1}}(c) \sigma_l(\rho_1(k))c^*$. \end{proof}

\noindent Next, we present the main result of this section. Some aspects of its proof build on ideas and techniques developed in \cite{CIOS1,CIOS4}. However, for the reader’s convenience, we include all the details here.

\begin{thm}\label{super+torsion} For every $j=1,2$ let $A_j$ be a nontrivial abelian group, $B_j$ a nontrivial icc subgroup of a hyperbolic group. Let $B_j\curvearrowright I_j$ be an action such that, for every $i\in I_j$, $\emph{Stab}_{B_j}(i)$ is amenable. Suppose that $G_j\in \WR (A_j,B_j\curvearrowright I_j)$ is a group with property (T). Assume that $C<G_1,G_2$ a common malnormal subgroup with Haagerup property. Consider the corresponding amalgamated free product $G= G_1\ast_C G_2$. Let $H$ be an arbitrary group and let $\theta\colon \emph{L}(G)^t\rightarrow \emph{L}(H)$ be a $*$-isomorphism for some $t>0$. If either \begin{enumerate}
    \item [(a)] $0< t\leq 1$, or
    \item [(b)] $H$ is torsion free 
\end{enumerate} then $G\cong H$ and $t=1$. Moreover, there exist a group isomorphism $\delta\colon G\rightarrow H$, a character $\eta\colon G\rightarrow\mathbb T$ and a unitary $w\in \emph{L}(H)$ such that $\theta(u_g)=\eta(g)wv_{\delta(g)}w^*$ for every $g\in G$, where $(u_g)_{g\in G}$ and $(v_h)_{h\in H}$ denote the canonical generating unitaries of $\emph{L}(G)$ and $\emph{L}(H)$, respectively.
\end{thm}

\begin{proof} We start by fixing some notations that will be used consistently throughout the proof. For $i=1,2$ let $\mathcal M_i =\mathcal L(G_i)$ and let $\mathcal C = \mathcal L(C)$. Withe these notations at hand we establish the following:

\begin{claim}\label{cornerofMiintoMjMk}
For every $i,j,k=1,2$ there exist projections $p_i^{j,k}\in \Delta(\mathcal M_i)'\cap p \mathscr M p $ with $\sum^2_{j,k=1}{p_i^{j,k}}=p$ and unitaries $u_i^{j,k}\in \mathscr M$ satisfying $u_i^{j,k} \Delta(\mathcal M_i) p_i^{j,k}( u_i^{j,k})^*\subseteq \mathcal M_j\bar\otimes \mathcal M_k\bar\otimes \mathbb M_{n}(\mathbb C).$
\end{claim}

\noindent \emph{Proof of Claim \ref{cornerofMiintoMjMk}.}  Since $\mathscr M$ is a II$_1$ factor after conjugating by a unitary in $\mathscr M$ we can assume that $p\in \mathcal C   \bar \otimes  \mathcal C    \bar \otimes  \mathbb D_n$. Fix $i\in \{1,2\}$. Next we argue that  $\Delta(\M_i) \nprec  \M \bar \otimes   \mathcal C  \bar\otimes \mathbb M_n(\mathbb C)$. Since $\M_i$ has property (T) and $\mathcal{C}$ has Haagerup property, if $\Delta(\M_i)\prec\M \bar{\otimes} \mathcal{C} \bar{\otimes} \mathbb M_n(\mathbb C)$ by \cite[Lemma 1]{HPV11} we must have $\Delta(\M_i)\prec \M \bar{\otimes} 1 \bar{\otimes} 1$, a contradiction to $\M_i$ being diffuse and Lemma \ref{comultiply}.


Now note that $\Delta(\mathcal M_i)\subseteq p\mathscr M p = p(\mathcal M   \bar\otimes  \mathcal \M_1   \bar\otimes  \mathbb M_n(\mathbb C))\ast_{\mathcal M  \bar\otimes \mathcal{C} \bar \otimes \mathbb M_n(\mathbb C)} (\mathcal M  \bar \otimes   \M_2  \bar\otimes  \mathbb M_n(\mathbb C))$ is a property (T) von Neumann algebra and $\M  \bar\otimes\M_l   \bar\otimes  \mathbb M_n(\mathbb C)$ is a II$_1$ factor, for $l=1,2$. Using \cite[Theorem 5.1]{IPP05} one can find projections $p^j_i\in \mathscr Z(\Delta(\M_i)'\cap p\mathscr M p)$ and unitaries $u_i^j\in\mathscr M$ such that $p^1_i+p_i^2=p$ and $u_i^{j} \Delta(\mathcal M_i) p_i^{j}( u_i^{j})^*\subseteq \mathcal M\bar\otimes \mathcal M_j\bar\otimes \mathbb M_{n}(\mathbb C).$ Repeating the same argument on the left tensor we get the desired conclusion. 
$\hfill\blacksquare$

Fix $i,j,k$ such that $p_i^{j,k}\neq 0$. To simplify the notations let $f= u_{i}^{j,k}p_i^{j,k}(u_i^{j,k})^*$. Notice that after replacing $\Delta$ by ${\rm Ad}(u_i^{j,k})\circ \Delta$, Claim \ref{cornerofMiintoMjMk} entails that 

\begin{equation}\label{inclusionMiintoMjMk}
    \Delta(\mathcal M_i) f\subseteq \mathcal M_j\bar\otimes \mathcal M_k\bar\otimes \mathbb M_{n}(\mathbb C)=\mathcal N.
\end{equation}

\noindent We now follow closely the proof of \cite[Theorem 4.1]{CIOS3}, as the map $\Delta$ is seen as an embedding of $\M_i$ into an amplification of $\M_j \overline{\otimes} \M_k$ by relation \eqref{inclusionMiintoMjMk}.

Let $\pi_i\colon G_i\rightarrow  B_i$ be the quotient homomorphism. For $g\in B_i$, fix $\widehat{g}\in G_i$ with $\pi_i(\widehat{g})=g$. 
Let $\mathbb D_n(\mathbb C)\subset\mathbb M_n(\mathbb C)$ be the subalgebra of diagonal matrices. For $1\leq i\leq n$, let $e_i={\bf 1}_{\{i\}}\in\mathbb D_n(\mathbb C)$.
Denote $$\text{$\P_i=\text{L}(A_i^{(I_i)})$,\;\;  $\sP=\text L(A_j^{(I_j)}) \bar{\otimes} \text L(A_k^{(I_k)}) \bar{\otimes} \mathbb D_n(\mathbb C)$\;\; and \;\;
$\Q=(\Delta(\P_i)f)'\cap f\mathcal N f$.}$$

\noindent We now continue with the following claim.

\begin{claim}\label{int1}$\Q\prec_{\mathcal N}^{\text s}\sP$.\end{claim}

\noindent{\textit{Proof of Claim \ref{int1}}.} We first prove that $\Delta(\P_i)f\prec_{\mathcal N}^{\text s}\sP$. Write $\mathcal N=(\M_j \bar{\otimes}  1 \bar{\otimes} \mathbb M_n(\mathbb C))\rtimes G_k$, using the trivial action of $G_k$. As $\ker(\pi_k)=A^{(I_k)}_k$, we have $(\M_j \bar{\otimes}  1 \bar{\otimes} \mathbb M_n(\mathbb C))\rtimes \ker(\pi_k)=\M_j \bar{\otimes} \P_k \bar{\otimes} \mathbb M_n(\mathbb C)$. Since  $\Delta(\M_i)f$ has property (T) and $\P_k$ is amenable, by \cite[Lemma 1]{HPV11} and Lemma \ref{comultiply}(a) we must have that $\Delta(\M_i)f\nprec_{\mathcal N} \M_j \bar{\otimes} \P_k \bar{\otimes} \mathbb M_n(\mathbb C)$. Therefore, since $\Delta(\P_i)f$ is amenable and $\Delta(\M_i)f\subset\sN_{f\mathcal N f}(\Delta(\P_i)f)''$ has property (T), by Theorem \ref{relativeT} we derive that $\Delta(\P_i)\prec_{\mathcal N}^{\text s}\M_j \bar{\otimes} \P_k \bar{\otimes} \mathbb M_n(\mathbb C)$.
Similarly, $\Delta(\P_i)f\prec_{\mathcal N}^{\text s}\P_j \bar{\otimes} \M_k \bar{\otimes} \mathbb M_n(\mathbb C)$. Combining these facts with \cite[Lemma 2.8(2)]{DHI16} gives that $\Delta(\P_i)f\prec_{\mathcal N}^{\text s}\P_j \bar{\otimes} \P_k \bar{\otimes} \mathbb M_n(\mathbb C)$, which proves our claim.

Next, we have that $\Delta(\P)f\nprec_{\mathscr M} \P_j \bar{\otimes}  1 \bar{\otimes}  \mathbb M_n(\mathbb C)$ and $\Delta(\P)f\nprec_{\mathscr M} 1 \bar{\otimes} \P_k \bar{\otimes}  \mathbb M_n(\mathbb C)$ by Lemma \ref{comultiply}(a). Since the action $B\curvearrowright I$ has amenable stabilizers, $\{i\in I\mid b\cdot i\not=i\}$ is infinite, for every $b\in B\setminus\{1\}$, by Lemma \ref{infinite}(a). Thus, using that $\Delta(\P_i)f\prec_{\mathcal N}^{\text s}\sP$, Corollary \ref{solidity} implies that $\Q$ is amenable. Finally, since $\Q$ is amenable and $\Delta(\M_i)f\subset\sN_{f\mathcal N f}(\Q)''$, repeating the first paragraph of the proof with $\Q$ instead of $\Delta(\P_i)f$ completes the proof of this step.
$\hfill\blacksquare$ 

\vspace{1mm}

\noindent As $\sP\subset\mathcal N$ is a Cartan subalgebra, there is a unitary $w\in \mathcal N$ such that $f\in w\mathscr P w^*$ and $w \mathscr P w^* f \subset f\mathcal N f$ is also a Cartan subalgebra. Moreover, one can see that Claim \ref{int1} further implies that $\mathcal Q \prec^{\rm s}_{f\mathcal N f} w \mathscr Pw^* f$. Combining this with Lemma \ref{conj1}, one can find a new unitary $u\in \mathcal N $ such that  $\Delta(\P_i)f\subset  u \sP u^*f\subset\Q$. Thus replacing $\Delta$ with $\text{Ad}(u^*)\circ\Delta$ and $f$ by $u^*f u$ we may assume that $f\in \sP$ and


\begin{equation}\label{deltaP}
    \Delta(\P_i)f\subset \sP f\subset\Q.
\end{equation}

\noindent If $g\in B_i$, then $\Delta(u_{\widehat{g}})f$ normalizes $\Delta(\P_i)f$ and thus $\Q$. Denote $\sigma_g=\text{Ad}(\Delta(u_{\widehat{g}})f)\in\text{Aut}(\Q)$. From $\widehat{g}\widehat{h}{\widehat{(gh)}}^{-1}\in\ker(\pi)=A_i^{(I_i)}$, we have $\Delta(u_{\widehat{g}})\Delta(u_{\widehat{h}})\Delta(u_{\widehat{gh}})^*\in\Delta(\P_i)$, for every $g,h\in B_i$. Since $\Delta(\P_i)f\subset\sZ(\Q)$, $\sigma =(\sigma_g)_{g\in B_i}$ defines an action of $B_i$ on $\Q$ which leaves $\Delta(\P_i)f$ invariant. Notice that the restriction of $\sigma$ to $\Delta(\P_i)f$ is conjugate to an action $B_i\curvearrowright \text{L}(A_i)^{I_i}$ built over $B_i\curvearrowright I_i$ and thus, it is free. Hence, Lemma \ref{conj2} yields an action $\beta=(\beta_g)_{g\in B_i}$ of $B_i$ on $\Q$ satisfying:

\begin{enumerate}
    \item[(i)] for every $g\in B_i$ we have that $\beta_g=\sigma_g\circ\text{Ad}(\omega_g)=\text{Ad}(\Delta(u_{\widehat{g}})f\omega_g)$, for some $\omega_g\in\sU(\Q)$,
    \item[(ii)] $\sP f$ is $\beta(B_i)$-invariant and the restriction of $\beta$ to $\sP f$ is free.
\end{enumerate}

\noindent Our next goal is to apply Theorem \ref{SOE}. Let $(Y_j,\nu_j)$ be the dual of $A_j$ with its Haar measure. Let $(X,\mu)=(Y_j^{I_j}\times Y_k^{I_k},\nu_j^{ I_j}\times\nu_k^{I_k})$ and $(\widetilde X,\widetilde\mu)=(X\times\mathbb Z/n\mathbb Z, \mu\times c)$, where $c$ is the counting measure of $\mathbb Z/n\mathbb Z$. Identify $\P_i=\text{L}^{\infty}(Y_i^{I_i})$ and $\sP=\text{L}^{\infty}(\widetilde X)$. Consider the action $B_j\curvearrowright^{\alpha_j} (Y_j^{I_j},\nu_k^{I_k})$ given by $\alpha_j(g)=\text{Ad}(u_{\widehat{g}})$, for all $g\in B_j$. Observe that $\alpha_j$ is conjugate to an action built over $B_j\curvearrowright I_j$. Let $ B_j\times B_k\curvearrowright^{\alpha}(X,\mu)$ be  given by $(g_1,g_2)\cdot (x_1,x_2)=(g_1\cdot x_1,g_2\cdot x_2)$, for all $g_1\in B_j,g_2\in B_k$ and $x_1\in Y_j^{I_j},x_2\in Y_k^{I_k}$. Denote by $\tilde \alpha$ the action $\tilde{\alpha}: B_j \times B_k\times\mathbb Z/n\mathbb Z\curvearrowright (\widetilde X,\widetilde\mu)$ given by $(g,a)\cdot (x,b)=(g\cdot x,a+b)$. Let $X_0\subset \widetilde X$ be a measurable set with $f=\textbf{1}_{X_0}$. Since $\sP f=\text{L}^{\infty}(X_0)$ is $\beta(B_i)$-invariant, we get a measure preserving action $B_i\curvearrowright^\beta (X_0,\widetilde\mu_{|X_0})$.

Since the restriction of $\beta$ to $\sP f$ is implemented by unitaries in $f\mathcal N f$ and we have that $\sR(\sP\subset\mathcal N)=\sR( B_j\times  B_k\times\mathbb Z/n\mathbb Z\curvearrowright^{\widetilde\alpha}\widetilde X)$, we deduce that

\begin{equation}\label{include}
    \beta( B_i)\cdot x\subset\widetilde\alpha( B_j\times B_k\times\mathbb Z/n\mathbb Z)\cdot x,  \text{for almost every $x\in X_0$.}
\end{equation}

\noindent Notice that $\alpha$ isomorphic to an action $B_j\times B_k\curvearrowright ((Y_j\times Y_k)^I,(\nu_j\times \nu_k)^I)$ built over $B_j\times B_k\curvearrowright I= I_j\sqcup I_k$ given by $(g_1,g_2)\cdot i=g_l\cdot i$ for $i\in I_l$ where $l=j$ or $k$. This, combined with the fact that $B_i$ has property (T) and \cite[Lemma 3.11]{CIOS3}, gives us a partition $X_0=\bigsqcup_{t=1}^lX_t$ into non-null measurable subsets, $l\in \mathbb N\cup\{\infty\}$, and finite index subgroups $S_t\leqslant B_i$ such that $X_t$ is $\beta(S_t)$-invariant and $\beta|_{S_t}$ on $X_t$ is weakly mixing for all $1\leq t\leq l$.

In order to apply Theorem \ref{SOE} to \eqref{include}, we establish the following claim.

\begin{claim}\label{intertwiningintoamenablestab} Fix $ r\in I_j$ (or $r\in I_k$). Then there is a sequence $(h_m)\subset S_t$ such that for every $s,t\in B_j$ (or $s,t\in B_k$) we have $\widetilde\mu(\{x\in X_0\mid \beta_{h_m}(x)\in\widetilde\alpha(sStab_{B_j}(r)t\times B_k\times\mathbb Z/n\mathbb Z)(x)\})\rightarrow 0$ (or $\widetilde\mu(\{x\in X_0\mid \beta_{h_m}(x)\in\widetilde\alpha(B_j\times sStab_{B_k}(r)t\times\mathbb Z/n\mathbb Z)(x)\})\rightarrow 0$).
\end{claim}

\noindent{\textit{Proof of Claim \ref{intertwiningintoamenablestab}}.} Let $G_0=\pi^{-1}(\text{Stab}_{B_k}(r)$). Then $G_0<G_k$ is an amenable subgroup. Furthermore, since $\pi_i(S_t)^{-1}$ has property (T), using Lemma \ref{comultiply}(a) we find a sequence $(k_m)\subset \pi_i(S_t)^{-1}$ such that for every $x,y\in\mathcal N$ we have

\begin{equation}\label{weak}
    \|E_{\M_j \bar{\otimes} \text{L}(G_0) \bar{\otimes} \mathbb M_n(\mathbb C)}(x\Delta(u_{k_m})fy)\|_2\rightarrow 0.
\end{equation}

\noindent We will show that $h_m=\pi_i(k_m)\in S_t$ satisfies the assertion of the claim. For that, we will first show that $\|E_{\M_j \bar{\otimes} \text{L}(G_0) \bar{\otimes} \mathbb M_n(\mathbb C)}(x\Delta(u_{\widehat{h_m}})f\omega_{h_m}y)\|_2\rightarrow 0$ as $m\to\infty$, for every $x,y\in(\mathcal{N})_1$. Observe that, since $\mathscr P$ is regular in $\mathcal N$, basic approximations show that it suffices to prove our claim only when $x, y\in \mathscr N_{\mathcal N}(\mathscr P)$, which will assume for the rest of the proof.

Notice that $k_m^{-1}\widehat{h_m}\in A_i^{(I_i)}$, $f\in\mathscr P$ and $\omega_{h_m}\in\sU(\Q)$, and so, we obtain that $\Delta(u_{\widehat{h_m}})f\omega_{h_m}\in\Delta(u_{k_m}) \sU(\Q)$. Let $\mathscr Z(\mathcal{Q})$ be the center of $\mathcal{Q}$, and notice it is contained in $\mathscr P\subseteq \mathcal{Q}$, as $\mathscr P\subset\mathcal{N}$ is a masa. Hence, there exist (countably many) projections $(r_\iota)_\iota\subset\mathscr{Z}(\mathcal{Q})$ with $\sum_\iota r_\iota=1$ such that $\mathcal{Q} r_\iota =\mathscr Z(\mathcal{Q})r_\iota \bar{\otimes} \mathbb M_{n_\iota}(\mathbb C)$, for some $n_\iota\in\mathbb N$. In particular, the inclusion $\mathscr Z(\mathcal{Q})r_\iota\subseteq \mathcal{Q}r_\iota$ admits a finite Pimsner-Popa basis and thus so is $\mathscr P r_\iota \subseteq \mathcal Q r_\iota$, for all $\iota$. Denote by $(v_{\iota,l})_{l=1}^{k_\iota}\subset \mathcal Q r_\iota$ a Pimsner-Popa basis for the inclusion $\mathscr Pr_\iota\subseteq\mathcal{Q}r_\iota$ so that for each $z\in \mathcal{Q}r_\iota$ we have $z=\sum_{l=1}^{k_\iota}v_{\iota,l}\mathbb E_{\mathscr P r_\iota}(v_{\iota,l}^*z)$.   

Fix $\epsilon>0$, and approximate $f\in\mathscr P$ by $f_{\epsilon}=\sum_{\iota\in F}fr_\iota$, for some finite set $F$. Notice that, for $x,y\in \mathscr N_{\mathcal N}(\mathscr P)$, we obtain the following uniform bound

\begin{equation*}
    \|\mathbb E_{\mathcal{M}_j\bar{\otimes}\text{L}(G_0)\bar{\otimes}\mathbb M_n(\mathbb C)}(x\Delta(u_{\widehat{h_m}})(f-f_\epsilon)\omega_{h_m}y)\|_2\leq \epsilon.
\end{equation*}

\noindent Thus, we only need to consider $\|\mathbb E_{\mathcal{M}_j\bar{\otimes}\text{L}(G_0)\bar{\otimes}\mathbb M_n(\mathbb C)}(x\Delta(u_{\widehat{h_m}})f_{\epsilon}\omega_{h_m} y)\|_2$. Write $\Delta(u_{\widehat{h_m}})f_{\epsilon}\omega_{h_m}=\sum_{\iota\in F}\Delta(u_{k_m})\Delta(u_{k_m^{-1}\widehat{h_m}})fr_\iota\omega_{h_m}$. Letting, $z_m=\Delta(u_{k_m^{-1}\widehat{h_m}})\omega_{h_m}\in\mathcal{Q}$, we may further rewrite $r_\iota z_m=\sum_{l=1}^{k_\iota}v_{\iota,l}\mathbb E_{\mathscr Pr_\iota}(v_{\iota,l}^*z_mr_\iota)$.  As  $\mathscr  P\subseteq \M_j \bar{\otimes} \text{L}(G_0) \bar{\otimes} \mathbb M_n(\mathbb C)$, the prior relation and $x,y\in \mathscr N_{\mathcal N}(\mathscr P)$ yield  

\begin{align}\label{v_h_m}
    \|\mathbb E_{\M_j \bar{\otimes} \text{L}(G_0) \bar{\otimes} \mathbb M_n(\mathbb C)}&(x\Delta(u_{\widehat{h_m}})f_{\epsilon}\omega_{h_m}y)\|_2\leq \|\sum_{\iota\in F}\sum_{l=1}^{k_\iota}\mathbb E_{\M_j \bar{\otimes} \text{L}(G_0) \bar{\otimes} \mathbb M_n(\mathbb C)}(x\Delta(u_{k_m})fv_{\iota,l}\mathbb E_{\mathscr Pr_\iota}(v_{\iota,l}^*z_mr_\iota)y)\|_2\nonumber\\
    &\leq\sum_{\iota\in F}\sum_{l=1}^{k_\iota}\|\mathbb E_{\M_j \bar{\otimes} \text{L}(G_0) \bar{\otimes} \mathbb M_n(\mathbb C)}(x\Delta(u_{k_m})fv_{\iota,l}y)\|_2,
\end{align}

\noindent Since neither $F$ nor $k_{\iota}$ depend on $m$, it follows that the last expression in \eqref{v_h_m} converges to zero by \eqref{weak}. Thus, $\lim_{m\to\infty}\|\mathbb E_{\mathcal{M}_j\bar{\otimes}\text{L}(G_0)\bar{\otimes}\mathbb M_n(\mathbb C)}(x\Delta(u_{\widehat{h_m}})f\omega_{h_m}y)\|_2\leq \epsilon$, for all $\epsilon>0$, yielding the claim.

On the other hand, we have that $\beta_{h_m}=\text{Ad}(\Delta(u_{\widehat{h_m}})f\omega_{h_m})$ and $\alpha(g_1,g_2)=\text{Ad}(u_{(\widehat{g_1},\widehat{g_2})})$, for every $(g_1,g_2)\in B_j\times B_k$. These facts imply that $\widetilde\mu(\{x\in X_0\mid  \beta_{h_m}(x)\in\widetilde{\alpha}(B_j\times s\text{Stab}_{B_k}(r)t\times\mathbb Z/n\mathbb Z)(x)\})=\|E_{\M_j\bar{\otimes} \text{L}(G_0) \bar{\otimes} \mathbb M_n(\mathbb C)}((u_{\widehat{s}}^*\otimes 1\otimes 1)\Delta(u_{\widehat{h_m}})f\omega_{h_m}(u_{\widehat{t}}^*\otimes 1\otimes 1))\|_2^2$ is less than $\epsilon$, as $m\rightarrow \infty$. As $\epsilon>0$ was arbitrary, this proves one of the assertion of this claim. The other follows similarly.
$\hfill\blacksquare$

\vspace{1mm} 

\noindent Since, by assumption, the centralizers $C_{B_j}(g)$ and $C_{B_k}(h)$, for $g\in B_j\setminus\{1\}$ and $h\in B_k\setminus\{1\}$, are amenable, the prior claim also holds after replacing stabilizers by centralizers.

Using Theorem \ref{SOE} we obtain injective homomorphisms $\varepsilon_t=(\varepsilon_{t,1},\varepsilon_{t,2}):S_t\to B_j\times B_k$ and $\varphi_t\in[\mathscr R(B_j\times B_k\times \mathbb Z/n\mathbb Z\curvearrowright \tilde{X})]$ such that $\varphi_t(X_t)=X\times\{t\}\equiv X$ and $\varphi_t\circ\beta_h|_{X_t}=\alpha_{\varepsilon_t(h)}\circ\varphi_t|_{X_t}$, for each $h\in S_t$. In particular, $(\tau_{j}\otimes\tau_k\otimes \text{Tr})(f)=\tilde\mu(X_0)=\sum_{t=1}^l\tilde{\mu}(X_t)=l\in\mathbb N$ and $l\leq n$. This shows that either $p_i^{j,k}$ are zero or have integer trace which in turn implies $(\tau\otimes\tau\otimes\text{Tr})(p)\in\mathbb N$; that is, $p=1_{\mathscr M}$ and $t=n$.

Denote by $p_t=\textbf{1}_{X_t}$, and let $u_t\in\sN_{\mathcal N}(\sP)$ be such that $u_tau_t^*=a\circ\varphi_t^{-1}$, for every $a\in\sP$. Then $u_tp_tu_t^*=1\otimes 1\otimes e_t$ and the last relation implies that we can find $(\zeta_{t,h})_{h\in B}\subset\sU(\P_j \overline{\otimes} \P_k)$ such that

\begin{equation}\label{conjug}
    \text{$u_t\Delta(u_{\widehat{h}})f\omega_hp_tu_t^*=\zeta_{t,h}u_{(\widehat{\varepsilon_{t,1}(h)},\widehat{\varepsilon_{t,2}(h)})}\otimes e_t$, for every $h\in S_t$.}
\end{equation}

\noindent After replacing $S_t$ by $S=\bigcap_{t=1}^l S_t$, we may assume $S_t=S$ for all $1\leq t\leq l$. Letting $K=\pi_i^{-1}(S)$ and $\M_0=\text{L}(K)$, we will show that we can find a projection $p_0\in f(\Delta(\M_0)'\cap \mathcal{N})f$ with $(\tau_j\otimes\tau_k\otimes\text{Tr})(p_0)=k\in\{1,\ldots,l\}$, homomorphisms $\delta:K\to G_j\times G_k$, $\rho:K\to\mathscr{U}_k(\mathbb C)$, a unitary $w\in \mathscr{U}(\mathcal{N})$ such that $\delta$ is injective, $wp_0w^*=1\otimes(\sum_{i=1}^ke_i)$ and $w\Delta(u_g)p_0w^*=u_{\delta(g)}\otimes \rho(g)$, for all $g\in K$.

\begin{claim}\label{conjugatedmaps} There exist $1\leq k\leq l$ and $\varepsilon:S\to B_j\times B_k$ such that after renumbering we have $p_0=\sum_{t=1}^kp_t\in f(\Delta(\M_0)'\cap\mathcal{N})f$ and we can take $\varepsilon_t=\varepsilon$ for every $1\leq t\leq k$.
\end{claim}

\noindent{\textit{Proof of Claim \ref{conjugatedmaps}}.} First we show that $\varepsilon_{t,1}(S)\leqslant B_j$ is nonamenable. Otherwise, $G_0=\pi_j^{-1}(\varepsilon_{t,1}(S))\leqslant G_j$ is amenable and from \eqref{conjug} we have $\Delta(\M_0)\prec \text{L}(G_0) \bar{\otimes} \M_k \bar{\otimes} \mathbb M_n(\mathbb C)$. Since $\M_0$ has property (T), \cite[Lemma 1]{HPV11} implies $\Delta(\M_0)\prec 1 \bar{\otimes} \M_k$, a contradiction to the fact that $\M_0$ is diffuse. Similarly, we obtain $\varepsilon_{t,2}(S)\leqslant B_k$ is a nonamenable subgroup.

Next we show that if $1\leq t_1,t_2\leq l$ satisfy $p_{t_1}\mathcal{Q}p_{t_2}\neq\{0\}$ then $\varepsilon_{t_1}$ is conjugated to $\varepsilon_{t_2}$. Let $x\in (\mathcal{Q})_1$ with $p_{t_1}xp_{t_2}\neq 0$, $x_h=\text{Ad}(\Delta(u_{\hat{h}})f\omega_h)(p_{t_1}x)=\beta_h(p_{t_1}x)\in (\mathcal{Q})_1$, and notice that $(\Delta(u_{\hat{h}})f\omega_hp_{t_1})xp_{t_2}=x_h(\Delta(u_{\hat{h}})f\omega_hp_{t_2})$, for $h\in S$. Using \eqref{conjug} we obtain

\begin{equation}\label{conjugxxh}
    (u_{t_1}^*(\zeta_{t_1,h}u_{\widehat{\varepsilon_{t_1}(h)}}\otimes e_{t_1})u_{t_1})xp_{t_2}=x_h(u_{t_2}^*(\zeta_{t_2,h}u_{\widehat{\varepsilon_{t_2}(h)}}\otimes e_{t_2})u_{t_2}), \text{for all }h\in S.
\end{equation}

\noindent For a finite set $F\subset B_j\times B_k$, denote by $P_F$ the orthogonal projection from $L^2(\mathcal{N})$ onto the $\|\cdot\|_2$-closed linear span of $\{u_g\otimes x:g\in (\pi_j\times \pi_k)^{-1}(F), x\in\mathbb M_n(\mathbb C)\}$. Since $(\zeta_{t_1,h})_{h\in S}, (\zeta_{t_2,h})_{h\in S}\subset\mathscr{U}(\mathcal{P}_j \overline{\otimes} \mathcal{P}_k)$ and $(x_h)_{h\in S}\subset(\mathcal{Q})_1$, using basic $\|\cdot\|_2$-approximations and \eqref{conjugxxh}, there exists a finite set $F\subset B_j\times B_k$ so that for all $h\in S$ we have

\begin{enumerate}
    \item[(a)] $\|(u_{t_1}^*(\zeta_{t_1,h}u_{\widehat{\varepsilon_{t_1}(h)}}\otimes e_{t_1})u_{t_1})xp_{t_2}-P_{F\varepsilon_{t_1}(h)F}((u_{t_1}^*(\zeta_{t_1,h}u_{\widehat{\varepsilon_{t_1}(h)}}\otimes e_{t_1})u_{t_1})xp_{t_2})\|_2<\|p_{t_1}xp_{t_2}\|_2/2$,
    \item[(b)] $\|x_h(u_{t_2}^*(\zeta_{t_2,h}u_{\widehat{\varepsilon_{t_2}(h)}}\otimes e_{t_2})u_{t_2})-P_{F\varepsilon_{t_2}(h)F}(x_h(u_{t_2}^*(\zeta_{t_2,h}u_{\widehat{\varepsilon_{t_2}(h)}}\otimes e_{t_2})u_{t_2}))\|_2<\|p_{t_1}xp_{t_2}\|_2/2$.
\end{enumerate}

\noindent Combining items (a) and (b) with \eqref{conjugxxh}, we derive that $F\varepsilon_{t_1}(h)F\cap F\varepsilon_{t_2}(h)F\neq\emptyset$, for each $h\in S$. Since $\varepsilon_{t_1}$ is injective and $S\leqslant B_i$ is a finite index subgroup, $\varepsilon_{t_1}(S)$ is an infinite group with property (T). For any $g\in B_j\setminus \{1\}$ (respectively $g\in B_k\setminus\{1\}$) $C_{B_j}(g)$ (respectively $C_{B_k}(g)$) is amenable, and for each $(g_1,g_2)\in B_j\times B_k\setminus\{(1,1)\}$, $C_{B_j\times B_k}((g_1,g_2))$ is contained in $C_{B_j}(g_1)\times B_k$ or $B_j\times C_{B_k}(g_2)$. Using that $\varepsilon_{t_1,1}(S), \varepsilon_{t_1,2}(S)$ are nonamenable, for $(g_1,g_2)\in B_j\times B_k\setminus\{(1,1)\}$ we obtain that $\{\varepsilon_{t_1}(h)(g_1,g_2)\varepsilon_{t_1}(h)^{-1}:h\in S\}$ is infinite. By \cite[Lemma 7.1]{BV13} there exists $g\in B_j\times B_k$ for which $\varepsilon_{t_1}(h)=g\varepsilon_{t_2}(h)g^{-1}$ for all $h\in S$. This proves our assertion that $\varepsilon_{t_1}$ and $\varepsilon_{t_2}$ are conjugate.
\vskip 0.07in
Let $\varepsilon=\varepsilon_1:S\to B_j\times B_j$. After renumbering, we may assume that $\varepsilon_1,\ldots,\varepsilon_k$ are conjugate to $\varepsilon$ while $\varepsilon_{k+1},\ldots,\varepsilon_l$ are not conjugate to $\varepsilon$, for some $1\leq k\leq l$. The prior assertion implies $p_t\mathcal{Q}p_{t'}=\{0\}$ for each $1\leq t\leq k$ and $k+1\leq t'\leq l$. Thus, $p_0:=\sum_{t=1}^kp_t$ belongs to the center of $\mathcal{Q}f$. As $p_0$ commutes with $\Delta(u_{\hat{h}})f\omega_h$ and $\omega_h\in \mathcal{Q}$, $p_0$ commutes with $\Delta(u_{\hat{h}})f$ for each $h\in S$. Since $p_t\in\mathscr Pf\subset\mathcal{Q}=\Delta(\mathcal{P}_i)'\cap\mathcal{N}$, $p_0$ also commutes with $\Delta(\mathcal{P}_i)f$. Therefore, $p_0\in f(\Delta(\M_0)'\cap\mathcal{N})f$, as wanted. Finally, for $1\leq t\leq k$ there exists $g_t\in B_j\times B_k$ with $\varepsilon_t(h)=g_t\varepsilon(h)g_t^{-1}$, for each $h\in S$, so after replacing $\varphi_t$ with $\alpha_{g_t^{-1}}\circ\varphi_t$, we may assume $\varepsilon=\varepsilon_t$ for every $1\leq t\leq k$.
$\hfill\blacksquare$

\vspace{1mm}

\noindent Let $u=\sum_{t=1}^ku_tp_k$ and $e=\sum_{t=1}^ke_t$. Notice that $u$ is a partial isometry with $uu^*=1\otimes e$ and $u^*u=p_0$. Let $\zeta_h=\sum_{t=1}^k\zeta_{t,h}\otimes e_t\in \mathscr{U}(\mathscr P(1\otimes 1\otimes e))$ and use \eqref{conjug} to obtain

\begin{equation*}
    u\Delta(u_{\hat{h}})f\omega_hp_0u^*=\zeta_h(u_{\widehat{\varepsilon(h)}}\otimes 1), \text{ for every }h\in S.
\end{equation*}

\noindent Identify $\mathcal{N}_1=(1\otimes 1\otimes e)\mathcal{N}(1\otimes 1\otimes e)$ with $\M_j \bar{\otimes} \M_k \bar{\otimes} \mathbb M_k(\mathbb C)$ and $\mathscr P_1=\mathscr P(1\otimes 1\otimes e)$ with $\mathcal{P}_j \bar{\otimes} \mathcal{P}_k \bar{\otimes}\mathbb D_k(\mathbb C)$. Consider the unital $*$-homomorphism $\Delta_1:\M_0\to\mathcal{N}_1$ given by $\Delta_1(x)=u\Delta(x)p_0u^*$ for $x\in M_0$. Denote by $\mathcal{Q}_1=\Delta_1(\mathcal{P}_i)'\cap\mathcal{N}_1$, and notice that $w_h:=u\omega_hp_0u^*\in \mathscr U(\mathcal{Q}_1)$. Moreover, from the prior equation we have

\begin{equation*}
    \Delta_1(u_{\hat{h}})w_h=\zeta_h(u_{\widehat{\varepsilon(h)}}\otimes 1), \text{ for every }h\in S.
\end{equation*}

\noindent From \eqref{deltaP}, $\mathscr{P}_1\subseteq\mathcal{Q}_1$. Since $(\Delta_1(u_{\hat{h}})w_h)_{h\in S}$ normalizes $\mathcal{Q}_1$ and $(\zeta_h)_{h\in S}\subset\mathscr{U}(\mathscr P_1)$, the equation above implies $(u_{\widehat{\varepsilon(h)}}\otimes 1)_{h\in S}$ normalizes $\mathcal{Q}_1$. Letting $\eta_h:=\zeta_h\text{Ad}(u_{\widehat{\varepsilon(h)}}\otimes 1)(w_h^*)\in\mathscr U(\mathcal{Q}_1)$, we obtain

\begin{equation}\label{conjugDelta1}
    \Delta_1(u_{\hat{h}})=\eta_h(u_{\widehat{\varepsilon(h)}}\otimes 1),\text{ for every }h\in S.
\end{equation}

\begin{claim}\label{Q1=PjPkT} $\mathcal{Q}_1=\mathcal{P}_j \bar{\otimes} \mathcal{P}_k \bar{\otimes} \mathcal{T}$, for some von Neumann subalgebra $\mathcal{T}\subseteq \mathbb M_k(\mathbb C)$.
\end{claim}

\noindent{\textit{Proof of Claim \ref{Q1=PjPkT}}.} We first show that $\mathcal{Q}_1\subseteq \mathcal{P}_j \bar{\otimes} \mathcal{P}_k \bar{\otimes} \mathbb{M}_k(\mathbb C)$. Denote the map $\varepsilon=(\varepsilon_1,\varepsilon_2)$ where $\varepsilon_l:S\to B_l$ with $l=j$ or $k$. Since $\varepsilon_1(S)$ and $\varepsilon_2(S)$ are nonamenable, there exists a sequence $(h_m)_m\subset S$ such that for each $g_1,g_2\in B_j\times B_k$ and $g_3\in B_j\times B_k\setminus\{(1,1)\}$ we have $g_1\text{Ad}(\varepsilon(h_m))(g_3)g_2\neq 1$ for $m$ large enough. 

Next we prove that $\|\mathbb E_{\mathcal{P}_j\bar{\otimes}\mathcal{P}_k}(a_1\text{Ad}(u_{\widehat{\varepsilon(h_m)}})(b)a_2)\|_2\to 0$ for every $a_1,a_2\in \M_j \bar{\otimes} \M_k$ and $b\in \M_j \bar{\otimes} \M_j\ominus(\mathcal{P}_j \bar{\otimes} \mathcal{P}_k)$. We check this for $a_1=u_{g}$, $a_2=u_{h}$, $b=u_s$ where $g,h\in G_j\times G_k$ and $s\in G_j\times G_k\setminus(A_j^{(I_j)}\times A_k^{(I_k)})$. Notice that $(\pi_j\times\pi_k)(s)\neq (1,1)$, so $(\pi_j\times \pi_k)(g\text{Ad}(\widehat{\varepsilon(h_m)})(s)h)=(\pi_j\times \pi_k)(g)\text{Ad}(\varepsilon(h_m))((\pi_j\times \pi_k)(s))(\pi_j\times \pi_k)(h)\neq (1,1)$ for $m$ large enough. Thus, $\mathbb E_{\mathcal{P}_j\bar{\otimes}\mathcal{P}_k}(a_1\text{Ad}(u_{\widehat{\varepsilon(h_m)}})(b)a_2)=0$ for $m$ large enough. Using basic approximations, this fact further implies

\begin{equation}\label{expectationP1to0}
    \lim_{m\to\infty}\|\mathbb E_{\mathscr{P}_1}(\text{Ad}(u_{\widehat{\varepsilon(h_m)}}\otimes 1)(b))\|_2=0, \text{ for all } b\in\mathcal{N}_1\ominus (\mathcal{P}_j \bar{\otimes} \mathcal{P}_j \bar{\otimes} \mathbb M_k(\mathbb C)).
\end{equation}

\noindent Let $\mathscr Z(\mathcal{Q}_1)$ be the center of $\mathcal{Q}_1$, and notice it is contained in $\mathscr P_1\subseteq \mathcal{Q}_1$. Moreover, $\mathcal Q_1\prec^s_{\mathcal{N}_1}\mathscr P_1$ by Claim \ref{int1}, and since $\mathscr P_1$ is abelian, it follows that $\mathcal{Q}_1$ is a type I von Neumann algebra. Hence, there exist projections $(r_{\iota})_{\iota}\subset\mathscr{Z}(\mathcal{Q}_1)$ with $\sum_{\iota} r_{\iota}=1$ such that $\mathcal{Q}_1r_{\iota}=\mathscr Z(\mathcal{Q}_1)r_{\iota} \overline{\otimes} \mathbb M_{n_{\iota}}(\mathbb C)$, for some $n_{\iota}\in\mathbb N$. In particular, the inclusion $\mathscr Z(\mathcal{Q}_1)r_{\iota}\subseteq \mathcal{Q}_1r_{\iota}$ admits a finite Pimsner-Popa basis, and therefore, so does the inclusion $\mathscr P_1r_{\iota}\subseteq\mathcal{Q}_1r_{\iota}$, for all ${\iota}$. Combining this with \eqref{expectationP1to0}, basic $\|\cdot\|_2$-approximations show that

\begin{equation*}
    \lim_{m\to \infty}\|\mathbb E_{\mathcal{Q}_1}(\text{Ad}(u_{\widehat{\varepsilon(h_m)}}\otimes 1)(b))\|_2=0, \text{ for all }b\in\mathcal{N}_1\ominus (\mathcal{P}_j \bar{\otimes} \mathcal{P}_k \bar{\otimes} \mathbb M_k(\mathbb C)).
\end{equation*}

\noindent To see that $\mathcal{Q}_1\subseteq \mathcal{P}_j \bar{\otimes} \mathcal{P}_k \bar{\otimes} \mathbb M_k(\mathbb C)$, let $a\in\mathcal{Q}_1$, denote by $c=\mathbb E_{\mathcal{P}_j\bar{\otimes}\mathcal{P}_k\bar{\otimes}\mathbb M_k(\mathbb C)}(a)$ and let $b=a-c$. Since $u_{\widehat{\varepsilon(h_m)}}\otimes 1$ normalizes $\mathcal{Q}_1$, $\text{Ad}(u_{\widehat{\varepsilon(h_m)}}\otimes 1)(a)\in \mathcal{Q}_1$, which means $\|\mathbb E_{\mathcal{Q}_1}(\text{Ad}(u_{\widehat{\varepsilon(h_m)}}\otimes 1)(a))\|_2=\|a\|_2$. By the prior equation, $\|\mathbb E_{\mathcal{Q}_1}(\text{Ad}(u_{\widehat{\varepsilon(h_m)}}\otimes 1)(b))\|_2\to 0$ which means $\|\mathbb E_{\mathcal{Q}_1}(\text{Ad}(u_{\widehat{\varepsilon(h_m)}}\otimes 1)(c))\|_2\to\|a\|_2$. Since $\|\mathbb E_{\mathcal{Q}_1}(\text{Ad}(u_{\widehat{\varepsilon(h_m)}}\otimes 1)(c))\|_2\leq \|c\|_2\leq\|a\|_2$, it means that $\|c\|_2=\|a\|_2$ and so $a=c\in \mathcal{P}_j \bar{\otimes} \mathcal{P}_k \bar{\otimes} \mathbb M_k(\mathbb C)$.

Therefore, $\mathcal{P}_j \bar{\otimes} \mathcal{P}_k \bar{\otimes} \mathbb D_k=\mathscr P_1\subseteq \mathcal{Q}_1\subseteq \mathcal{P}_j \bar{\otimes} \mathcal{P}_k \bar{\otimes} \mathbb M_k(\mathbb C)$. Since $\mathcal{P}_j \bar{\otimes} \mathcal{P}_k=\text{L}^{\infty}(X,\mu)$, we can desintegrate $\mathcal{Q}_1=\int_X^{\oplus}\mathcal{T}_xd\mu(x)$, where $(\mathcal{T}_x)_x$ is a measurable field of von Neumann subalgebras of $\mathbb{M}_k(\mathbb C)$ containing $\mathbb D_k$ for almost every $x\in X$. Denote by $\gamma_h=\text{Ad}(u_{\widehat{\varepsilon(h)}})\in \text{Aut}(\mathcal{P}_j \bar{\otimes} \mathcal{P}_k)$, for $h\in S$. Since $(u_{\widehat{\varepsilon(h)}}\otimes 1)_{h\in S}$ normalizes $\mathcal{Q}_1$, we obtain $\mathcal{T}_{\gamma_h(x)}=\mathcal{T}_x$ for every $h\in S$ and almost every $x\in X$. Notice that $\gamma$ is built over the action $S\curvearrowright I=I_j\sqcup I_k$ given by $s\cdot i=\varepsilon_l(s)\cdot i$ for $i\in I_l$, where $l=j$ or $k$. Since $\varepsilon$ is injective, $S$ has property (T) and $\text{Stab}_{B_l}(i)$, for $i\in I_l$ with $l$ either $j$ or $j$, is amenable, the action $S\curvearrowright I$ has infinite orbits. This implies $\gamma$ is weakly mixing, hence ergodic. Thus, we can find a subalgebra $\mathcal{T}\subset \mathbb M_k(\mathbb C)$ with $\mathcal{T}=\mathcal{T}_x$ for almost every $x$. This shows that $\mathcal{Q}_1=\mathcal{P}_j \bar{\otimes} \mathcal{P}_k \bar{\otimes} \mathcal{T}$.
$\hfill\blacksquare$

\vspace{1mm}

\noindent Let $\mathscr U=\mathscr U(\mathcal{T})/\mathscr U(\mathscr Z(\mathcal{T}))$, where $\mathscr Z(\mathcal{T})$ is the center of $\mathcal{T}$, and denote by $q:\mathscr U(\mathcal{T})\to \mathscr U$ the quotient map. We continue with the following claim.

\begin{claim}\label{Delta1intocenterT} There exists maps $\xi:S\to\mathscr U(\mathcal T)$, $\nu:S\to\mathscr U(\mathscr Z(\mathcal{Q}_1))$ such that $S\ni h\mapsto q(\xi_h)\in\mathscr U$ is a homomorphism and after replacing $\Delta_1$ with ${\rm Ad}(u)\circ\Delta_1$, for some $u\in\mathscr U(\mathcal{Q}_1)$, we obtain $\Delta_1(u_{\hat{h}})=\nu_h(u_{\widehat{\varepsilon(h)}}\otimes\xi_h)$, for all $h\in S$.
\end{claim}

\noindent{\textit{Proof of Claim \ref{Delta1intocenterT}}.} Note that $(\text{Ad}\Delta_1(u_{\hat{h}}))_{h\in S}$ and $(\gamma_h\otimes \text{Id}_{\mathcal{T}})_{h\in S}$ define actions of $S$ on $\mathcal{Q}_1$. Combining this with \eqref{conjugDelta1} gives that

\begin{equation*}
    \eta_{h_1h_2}^*\eta_{h_1}(\gamma_{h_1}\otimes\text{Id})(\eta_{h_2})\in\mathscr Z(\mathcal{Q}_1)=\mathcal{P}_j \bar{\otimes} \mathcal{P}_k \bar{\otimes} \mathscr Z(\mathcal{T}).
\end{equation*}

\noindent Viewing $\eta$ as a measurable function $\eta:X\to\mathscr U(\mathcal{T})$, the prior equation rewrites as $\eta_{h_1h_2}(x)^*\eta_{h_1}(x)\eta_{h_2}(\gamma_{h_1^{-1}}(x))\in\mathscr U(\mathscr Z(\mathcal{T}))$ for every $h_1,h_2\in S$ and for almost every $x\in X$. Let $c:S\times X\to \mathscr U$ be given by $c(h,x)=q(\eta_h(x))$, and notice that it is a 1-cocycle for $\gamma$. Since $\mathscr U(\mathscr Z(\mathcal{T}))$ is a closed central subgroup of the compact Polish group $\mathscr U(\mathcal T)$, $\mathscr U$ is a compact Polish group with respect to the quotient topology. In particular, it is a $\mathscr U_{\text{fin}}$ group \cite[Lemma 2.7]{Po05}. Since $S$ has property (T), $\gamma$ is built over $S\curvearrowright I$ and $S\curvearrowright I$ has infinite orbits, \cite[Theorem 3.5]{CIOS3} implies $c$ is cohomologous to a homomorphism $\psi:S\to \mathscr U$.

Let $u:X\to \mathscr U(\mathcal T)$ be a measurable map satisfying $q(u(x))c(h,x)q(u(\gamma_{h^{-1}}(x)))^{-1}=\psi_h$, for all $h\in S$ and almost every $x\in X$. Let $\xi:S\to\mathscr U(\mathcal T)$ be such that $q(\xi_h)=\psi_h$. Thus, we find a measurable map $\nu_h:X\to\mathscr U(\mathscr Z(\mathcal T))$ such that $u(x)\eta_h(x)u(\gamma_{h^{-1}}(x))^{-1}=\nu_h(x)\xi_h$ for all $h\in S$ and almost every $x\in X$. Equivalently, $u\in\mathscr U(\mathcal{Q}_1)$ and $\nu:S\to\mathscr U(\mathscr Z(\mathcal T))$ satisfy $u\eta_h(\gamma_h\otimes\text{Id})(u^*)=\nu_h(1\otimes\xi_h)$, for every $h\in S$. Thus, $u\Delta_1(u_{\hat{h}})u^*=u\eta_h(u_{\widehat{\varepsilon(h)}}\otimes 1)u^*=\nu_h(u_{\widehat{\varepsilon(h)}}\otimes\xi_h)$, for all $h\in S$, as wanted.
$\hfill\blacksquare$

\vspace{1mm}

\noindent Take $g\in K=\pi_i^{-1}(S)$ and let $h=\pi_i(g)\in S$. Then $a=g\widehat{h^{-1}}\in A_i^{(I_i)}$ and $\Delta_1(u_a)\in\mathscr Z(\mathcal Q_1)$. Let $k_g=\Delta_1(u_a)\nu_h\in\mathscr{Z}(\mathcal{Q}_1)$ and notice that

\begin{equation}\label{Delta1onK}
    \Delta_1(u_g)=k_g(u_{\widehat{\varepsilon(\pi_i(g))}}\otimes\xi_{\pi_i(g)}),\text{ for all }g\in K.
\end{equation}

\begin{claim}\label{kg} There exists $c:K\to \mathcal{P}_j \bar{\otimes} \mathcal{P}_k$, $e:K\to\mathscr Z(\mathcal{T})$ and $z\in\mathscr U(\mathscr Z(\mathcal Q_1))$ such that $k_g=(u_{c_g}\otimes e_g)z^*(\gamma_{\pi_i(g)}\otimes {\rm Id})(z)$.
\end{claim}

\noindent{\textit{Proof of Claim \ref{kg}}.} Define $s:K\times K\to A_j^{(I_j)}\times A_k^{(I_k)}$ by $s_{g,h}=\widehat{\varepsilon(\pi_i(g))}\widehat{\varepsilon(\pi_i(h))}\widehat{\varepsilon(\pi_i(gh))}^*$, and $d:K\times K\to\mathscr U(\mathscr Z(\mathcal T))$ by $d_{g,h}=\xi_{\pi_i(gh)}\xi_{\pi_i(h)}^*\xi_{\pi_i(g)}^*$. From the prior equation \eqref{Delta1onK}, we obtain

\begin{equation}\label{sgh}
    u_{s_{g,h}}\otimes 1=k_{gh}k_g^*((\gamma_{\pi_i(g)}\otimes\text{Id})(k_h))^*(1\otimes d_{g,h}),\text{ for every }g,h\in K.
\end{equation}

\noindent Notice, moreover, the map $K\times K\ni (g,h)\mapsto u_{s_{g,h}}\in\mathscr U(\mathcal P_j \overline\otimes \mathcal{P}_k)$ is a 2-cocycle for the action $K\curvearrowright^{\gamma\circ\pi_i}\mathcal P_j \overline\otimes \mathcal{P}_k$, while the map $K\times K\ni(g,h)\mapsto d_{g,h}\in\mathscr U(\mathscr Z(\mathcal T))$ is a 2-cocycle for the trivial action. Since $\gamma$ is built over $S\curvearrowright I$, $\gamma\circ\pi_i$ is an action built over $K\curvearrowright I$ given by $k\cdot m=\pi_i(k)\cdot m$, for $k\in K$ and $m\in I$. Since the action $K\curvearrowright I$ has infinite orbits, $\gamma\circ\pi_i$ is weakly mixing. If $n\in\mathbb N$, then the action $(\gamma\circ\pi_i)^{\otimes n}$ is built over the diagonal product $K\curvearrowright I^n$, which has infinite orbits. Since $K$ has property (T), \cite[Theorem 3.5]{CIOS3} implies that $(\gamma\circ\pi_i)^{\otimes n}$ is $\mathscr{U}_{\text{fin}}$-cocycle superrigid. By \cite[Theorem 3.7]{CIOS3}, there exists $c:K\to A_j^{(I_j)}\times A_k^{(I_k)}$, $e:K\to\mathscr U(\mathscr Z(\mathcal{T}))$ satisfying $u_{s_{g,h}}=u_{c_{gh}}u_{c_g}^*\gamma_{\pi_i(g)}(u_{c_h})^*$ and $d_{g,h}=e_{gh}^*e_ge_h$ for every $g,h\in K$. From \eqref{sgh}, we can see that $S\ni g\mapsto k_g^*(u_{c_g}\otimes e_g)\in\mathscr U(\mathscr Z(\mathcal{Q}_1))$ is a 1-cocycle for $((\gamma\circ\pi_i)\otimes \text{Id}):K\curvearrowright \mathcal{P}_j \bar{\otimes} \mathcal{P}_k \bar{\otimes} \mathscr Z(\mathcal T)$. Since $\gamma\circ\pi_i$ is $\mathscr U_{\text{fin}}$ superrigid, we obtain the desired $z\in\mathscr U(\mathscr Z(\mathcal{Q}_1))$.
$\hfill\blacksquare$

\vspace{2mm}

\noindent It follows that $\Delta_1(u_g)=z^*(u_{c_g\widehat{\varepsilon(\pi_i(g))}}\otimes \xi_{\pi_i(g)}e_g)z$. The maps $\delta:K\to G_j\times G_k$ given by $\delta(g)=c_g\widehat{\varepsilon(\pi_i(g))}$ and $\rho:K\to\mathscr U(\mathcal{T})$ given by $\rho(g)=\xi_{\pi_i(g)}e_g$ are homomorphisms with 

\begin{equation}\label{this rel}
    \Delta_1(u_g)=z^*(u_{\delta(g)}\otimes \rho(g))z, \text{ for every } g\in K.
\end{equation} 

\noindent By construction $A_i^{(I_i)}\subset K$ and $\delta(A_i^{(I_i)})\subset A_j^{(I_j)}\times A_k^{(I_k)}$. Moreover, if $g\in\ker\delta$, then we have $\Delta_1(u_g)=z^*(1\otimes\rho(g))z\in z^*(1\otimes\mathbb M_k(\mathbb C))z$. This implies $\ker\delta$ is finite. Since $K$ is icc, it must be the case that $\delta$ is injective. This finishes the inductive step.

From here one we split the proof in two parts. First we assume case (a) i.e. $t\leq 1$. This automatically implies that $n=1$, $f=1$ and $\mathcal T =\mathbb C$. Moreover, we have $\Delta=\Delta_0$. In this case, relation \eqref{this rel} implies the existence of a finite index subgroup $K\leqslant G_i$, a unitary $z\in \mathcal M \bar\otimes \mathcal M$, a monomorphism $\delta=(\delta_1,\delta_2): K \ra G_j \times G_k$ and a multiplicative character $\mu: K \ra \mathbb T$ such that 

\begin{equation}\label{grouphomeq}
    \Delta(u_g) = \mu(g) z u_{\delta(g)}z^*, \text{ for all }g\in K.
\end{equation}

\vspace{0.5mm}

\begin{claim}\label{grouphom} Relation \eqref{grouphomeq} holds for $K=G_i$.
\end{claim}

\noindent\textit{Proof of Claim \ref{grouphom}.} Passing to a finite index subgroup we can assume that $K\leqslant G_i$ is normal. Now fix $g\in G_i$ and $h\in K$. Using relation \eqref{grouphomeq} for $h, ghg^{-1}\in K$ we see that

\begin{equation*}
    \mu(h)\Delta(u_g) z u_{\delta(h)}z^*\Delta(u_{g^{-1}})= \Delta (u_g)\Delta(u_h)\Delta(u_{g^{-1}})=\Delta(u_{ghg^{-1}})= \mu(ghg^{-1}) zu_{\delta(ghg^{-1})} z^*.
\end{equation*}
    
\noindent Letting $y:= z^* \Delta(u_g)z$ and $t_h:= \overline{\mu(h)}\mu(ghg^{-1})\in \mathbb T$, the previous equation implies that $ y u_{\delta(h)} =t_h  u_{\delta(ghg^{-1})} y$ for all $h\in K$. Since $K$ is infinite property (T) and the centralizers in $G_j, G_k$ are amenable, for every finite index subgroup $K_0\leqslant K$ we have $C_{G_j} (\delta_1(K_0))=\{1\}$ and $C_{G_k}(\delta_2(K_0))=\{1\}$. In particular, we have $C_{G_j \times G_k}(\delta(K_0))=\{1\}$. Using Proposition \ref{singlesupport} we get that $y=z^*\Delta(u_g)z\in\mathbb T(u_{(h,l)})_{(h,l)\in G_j\times G_k}$ for all $g\in G_i$. Thus, $\mu$ and $\delta$ extend to homomorphisms $\mu:G_i\to\mathbb T$ and $\delta:G_i\to G_j\times G_k$ with $\Delta(u_g)=\mu(g)zu_{\delta(g)}z^*$, for every $g\in G_i$. $\hfill\blacksquare$

\vspace{1mm}

\begin{claim}\label{j=k} In Claim \ref{grouphom}, we may take $j=k$ and $\delta_1=\delta_2$.
\end{claim}

\noindent\textit{Proof of Claim \ref{j=k}.}  
Since $\zeta\circ\Delta=\Delta$, where $\zeta$ is the flip map, from \eqref{grouphomeq}, we obtain

\begin{equation*}
    \zeta(z^*)z(u_{\delta_1(g)}\otimes u_{\delta_2(g)})=(u_{\delta_2(g)}\otimes u_{\delta_1(g)})\zeta(z^*)z, \text{ for every } g\in G_i.
\end{equation*}

\noindent Arguing as in the proof of the previous claim, we obtain that for every finite index subgroup $K_0 \leqslant G_i$ we have that $C_{G_j \times G_k}(\delta(K_0) )=\{1\}$. Using Proposition \ref{singlesupport}, one can find $c\in \mathbb T$ and $(h,l) \in G_j\times G_k$ with $\zeta(z^*)(z)= c u_h \otimes u_l$, so that ${\rm Ad} (h)\circ \delta_1=\delta_2$ and ${\rm Ad} (l)\circ \delta_2=\delta_1$. Replacing $z$ by $z (1\otimes u_h)$ we get $j=k$ and we can take $\delta_1 =\delta_2=:\upsilon$. 
$\hfill\blacksquare$

\vspace{1mm}

\noindent Since $i$ was arbitrary, we may now describe the comultiplication on both $\text{L}(G_1)$ and $\text{L}(G_2)$. In the claim below, we show that these descriptions can be chosen to be the same.

\vspace{1mm}

\begin{claim}\label{grouphomtotal} There exist a unitary $y\in \mathcal M\bar \otimes \mathcal M$ and a monomorphism $\gamma=(\gamma_1,\gamma_2): G \ra G \times G$ and a multiplicative character $\mu: G \ra \mathbb T$  such that  

\begin{equation}\label{untwist}
    \Delta(u_g) = \mu(g) y u_{\gamma(g)}y^*, \text{ for all }g\in G.
\end{equation}
\end{claim}

\noindent\textit{Proof of Claim \ref{grouphomtotal}.}  Using Claim \ref{grouphom} we can find unitaries $z_i\in \mathcal M\bar \otimes \mathcal M$, a monomorphism $\delta_i=(\upsilon_i,\upsilon_i): G_i \ra G_{j_i} \times G_{j_i}$ and a multiplicative character $\mu_i: G_i \ra \mathbb T$ satisfying 

\begin{equation}\label{completeuntwist1}
    \Delta(u_g) = \mu_i(g) z_i u_{\delta_i(g)}z^*_i, \text{ for all }g\in G_i.
\end{equation}

\noindent Thus, for every $g\in C \leqslant G_1 \cap G_2$ we have $\mu_1(g) z_1 u_{\delta_1(g)}z^*_1 = \Delta(u_g)=\mu_2(g) z_2 u_{\delta_2(g)}z^*_2 $ and hence if we let $z = z_2^*z_1$ and $t(g)= \overline{\mu_2(g)}\mu_1(g)$ we further have that $t(g) z u_{\delta_1(g)}  = u_{\delta_2(g)}z^*$ for all $g\in C$. 

Next we will argue that for every finite index subgroup $C_0\leqslant C$ we have that $C_{G \times G}(\delta_1 (C_0))=\{1\}$. Since by assumption for each $i=1,2$ and every $1\neq h\in G_i$ the centralizers $C_{G_i}(h)$ are amenable, the same holds for $G$ \cite[Theorem 1]{KS07}. Now since $C$ is nonamenable then so is $C_0$ and also $\delta_1(C_0)$, which means $C_{G\times G}(\delta_1(C_0))=\{1\}$. Using Proposition \ref{singlesupport} one can find scalars $c\in \mathbb C$ and $(h,l)\in G\times G$ such that  $z=c u_h\otimes u_l$. In particular, after conjugating $\delta_2$ by a inner group homomorphism of $G \times G$, we can assume without any loss of generality that in equations \eqref{completeuntwist1} we have $z_1 =z_2=:y$. However, these relations then automatically imply that the existence of an injective group homomorphism $\gamma =(\gamma_1,\gamma_2): G\ra G \times G$ and a multiplicative character $\mu: G \ra \mathbb T$ satisfying \eqref{untwist}.
$\hfill\blacksquare$

\vspace{1mm}

\begin{claim}\label{positiveheight} $h_H(G)>0$.
\end{claim}

\noindent\textit{Proof of Claim \ref{positiveheight}.} Fix $\epsilon>0$ and let $F_{\epsilon}=F_1\times F_2\subset H\times H$ be a finite set with $y_{\epsilon}\in \M\overline{\otimes}\M$ supported on $F_{\epsilon}$ with $\|y-y_{\epsilon}\|_2<\epsilon$. In that case, letting $u_g=\sum_h\tau(u_gv_h^*)v_h$, $y_{\epsilon}=\sum_{f\in F_{\epsilon}}\tau(y_{\epsilon}v_f^*)v_{(f_1,f_2)}$ and $u_{\gamma(g)}=\sum_{h_1,h_2}\tau(u_{\gamma_1(g)}v_{h_1}^*)\tau(u_{\gamma_2(g)}v_{h_2}^*)v_{(h_1,h_2)}$, we obtain

\begin{align*}
    \|y&\|_2^2=|\langle \Delta(u_g)yu_{\gamma(g^{-1})},y\rangle|\leq \epsilon(\|y\|_2+\|y_{\epsilon}\|_2)+|\langle \Delta(u_g),y_{\epsilon}u_{\gamma(g)}y_{\epsilon}^*\rangle|\\
    &\leq \epsilon(2\|y\|_2+\epsilon)+\sum_{h,h_1,h_2\in H, f,k\in F_{\epsilon}}|\tau(y_{\epsilon}v_{f^{-1}})\tau(y_{\epsilon}v_k)\tau(u_{g}v_{h^{-1}})\tau(u_{\gamma_1(g)}v_{h_1^{-1}})\tau(u_{\gamma_2(g)}v_{h_2}^{-1})|\delta_{h,f_1h_1k_1}\delta_{h,f_2h_2k_2}\\
    &\leq \epsilon(2\|y\|_2+\epsilon)+\|y_{\epsilon}\|_2^2\cdot |F_{\epsilon}|^2\cdot h_H(u_g)\leq \epsilon(2\|y\|_2+\epsilon)+(\|y\|_2+\epsilon)^2|F_{\epsilon}|^2\cdot h_H(u_g).
\end{align*}

\noindent Hence, for all $g\in G$, we have $h_H(u_g)\geq \frac{\|y\|_2^2-\epsilon(2\|y\|_2+\epsilon\|_2)}{(\|y\|_2+\epsilon)^2\cdot|F_{\epsilon}|^2}$. Taking $\epsilon>0$ small enough, $h_H(G)>0$.
$\hfill\blacksquare$

\vspace{2mm}

\noindent From \cite[Theorem 4.1]{KV15} we obtain a unitary $w\in\text{L}(H)$, a character $\mu:G\to\mathbb T$ and a group isomorphism $\rho:G\to H$ with $u_g=\mu(g)wv_{\rho(g)}w^*$. One can alternatively use the following claim to finish the theorem in this case. We write it here as we will use a version of it in Claim \ref{grouprep}.

\vspace{1mm}

\begin{claim}\label{deltaid} $\gamma(g)=(g,g)$, for every $g\in G$.
\end{claim}

\noindent\textit{Proof of Claim \ref{deltaid}.} By repeating Claim \ref{j=k}, we may assume $\gamma(g)=(\upsilon(g),\upsilon(g))$, for every $g\in G$. Now, since $\Delta$ satisfies $(\Delta\otimes\text{id})\circ \Delta=(\text{id}\otimes\Delta)\circ\Delta$, we obtain that

\begin{align*}
    (1\otimes y^*)(\text{id}\otimes\Delta)(y^*)(\Delta\otimes\text{id}&)(y)(y\otimes 1)(u_{\upsilon(\upsilon(g))}\otimes u_{\upsilon(\upsilon(g))}\otimes u_{\upsilon(g)})\\
    &=(u_{\upsilon(g)}\otimes u_{\upsilon(\upsilon(g))}\otimes u_{\upsilon(\upsilon(g))})(1\otimes y^*)(\text{id}\otimes\Delta)(y^*)(\Delta\otimes \text{id})(y)(y\otimes 1),
\end{align*}

\noindent for all $g\in G$. By Proposition \ref{singlesupport}, there exists $h\in G$ with $\upsilon(g)=hgh^{-1}$ for every $g\in G$. Replacing $y$ by $y(u_h\otimes u_h)$ gives $\Delta(u_g)=y(u_g\otimes u_g)y^*$, for every $g\in G$. 
$\hfill\blacksquare$

\vspace{1mm}

\noindent By \cite[Theorem 3.4]{IPV10}, there exists a unitary $w\in\text{L}(H)$, a character $\mu:G\to\mathbb T$ and a group isomorphism $\rho:G\to H$ with $u_g=\mu(g)wv_{\rho(g)}w^*$. This yields our conclusion in case (a).

Now assume case (b); that is, $H$ is torsion free. First we prove the following claim.

\begin{claim}\label{trivialrelcom} For every irreducible von Neumann subalgebra $\mathcal S\subseteq \M$ we have $\Delta(\mathcal S)'\cap (\mathcal M \bar\otimes \mathcal M \bar{\otimes} \mathbb M_n(\mathbb C))=\mathbb C 1$. 
\end{claim}

\noindent\textit{Proof of Claim \ref{trivialrelcom}.}
First we argue that for every $h\neq 1$ we have $\mathcal S\nprec {\rm L}(C_H(h))$. Indeed, if $\mathcal S\prec {\rm L}(C_H(h))$, passing to intertwining of relative commutants we obtain ${\rm L} (\langle h\rangle ) \subset {\rm L}(C_H(h))'\cap {\rm L}(H)\prec \mathcal S' \cap (\mathcal M \bar\otimes \mathbb M_n(\mathbb C))=(\mathcal S'\cap \mathcal M) \bar\otimes \mathbb M_n(\mathbb C)=\mathbb M_n(\mathbb C)$. However, by \cite[Proposition 2.6]{CdSS15}, this further implies that $\langle h\rangle $ is a finite group, contradicting that $H$ is torsion free. Using Proposition \ref{comultiply}(d) we obtain $\Delta(\mathcal S)'\cap \mathcal M\overline{\otimes}\M\overline{\otimes}\mathbb M_n(\mathbb C) = \Delta(S'\cap\mathcal M) =\mathbb C1$,  as $\mathcal S\subseteq \mathcal M$ is assumed to be irreducible. $\hfill\blacksquare$

\vspace{1mm}

\noindent Since $\mathcal M_i \subset \mathcal M$ is irreducible, we have $\Delta(\M_i)'\cap \mathcal M \bar\otimes \mathcal M \bar\otimes \mathbb M_n(\mathbb C)=\mathbb C1$, and hence, $f=1$ and $p_i^{j,k}=1$ in Claim \ref{cornerofMiintoMjMk}. Similarly, as $K\leqslant G_i$ has finite index one can see that $\mathcal M_0=\text{L}(K) \subseteq \mathcal M$ is also irreducible. Thus $\Delta(\M_0)'\cap \mathcal M \bar\otimes \mathcal M \bar\otimes \mathbb M_n(\mathbb C)=\mathbb C1$ and $p_0=1$ in Claim \ref{conjugatedmaps}. Altogether, relations \eqref{this rel} show the existence of a finite index subgroup $K\leqslant G_i$, a unitary $z\in \mathcal M \bar\otimes \mathcal M$, a monomorphism $\delta=(\delta_1,\delta_2): K \ra G_j \times G_k$ and a unitary representation $\rho: K \ra \mathscr U(\mathbb C^n)$ such that 

\begin{equation}\label{grouphomeq1}
    \Delta(u_g) = z(u_{\delta(g)}\otimes \rho(g))z^*, \text{ for all }g\in K.
\end{equation}

\noindent Next we prove similar statements as in case (a), leaving some details to the reader.

\vspace{1mm}

\begin{claim}\label{grouprep} Relation \eqref{grouphomeq1} holds for $K=G_i$, and we can take $\delta_1=\delta_2=:\upsilon_i$.
\end{claim}

\noindent\textit{Proof of Claim \ref{grouprep}.} The first part is similar to the proof of Claim \ref{grouphom}, and for the second part, we follow Claim \ref{deltaid} and \cite[Theorem 1.3 Step 6]{CIOS1}. From \eqref{Delta}, 

\begin{equation*}
    \Delta_0(u_g\otimes 1)=U\psi(z\otimes 1)(u_{\delta_1(g)}\otimes\rho(g)\otimes u_{\delta_2(g)}\otimes 1)\psi(z^*\otimes 1)U^*.
\end{equation*}

\noindent Since $\zeta\circ\Delta_0=\Delta_0$, where $\zeta$ is the flip map, letting $V=\zeta(\psi(z^*\otimes 1)U^*)U\psi(z\otimes 1)$ we obtain

\begin{equation*}
    V(u_{\delta_1(g)}\otimes\rho(g)\otimes u_{\delta_2(g)}\otimes 1)=(u_{\delta_2(g)}\otimes 1\otimes u_{\delta_1(g)}\otimes\rho(g))V,\text{ for all }g\in G_i.
\end{equation*}

\noindent By Proposition \ref{singlesupport}, we obtain $V=cu_{(h,l)}$ for some unitary $c\in\mathbb M_n(\mathbb C)\otimes\mathbb M_n(\mathbb C)$ and some $(h,l)\in G\times G$. This shows that $\text{Ad}(h)\circ\delta_1=\delta_2=:\upsilon_i$. Replacing $z$ by $z(u_h\otimes 1\otimes 1)$ gives $\Delta(u_g)=z(u_{\upsilon_i(g)}\otimes u_{\upsilon_i(g)}\otimes\rho(g))z^*$, for every $g\in G_i$. 
$\hfill\blacksquare$

\vspace{2mm}

\noindent To summarize, the prior claim and relation \eqref{grouphomeq1} imply existence of a unitary $z_i\in \mathcal M \bar\otimes \mathcal M\bar\otimes \mathbb M_n(\mathbb C)$, a monomorphism $\delta_i=(\upsilon_i,\upsilon_i): G_i \ra G_{j_i} \times G_{k_i}$ and a unitary representation  $\rho_i: G_i \ra \mathscr U(\mathbb C^n)$ such that 

\begin{equation*}
    \Delta(u_g) = z_i(u_{\delta_i(g)}\otimes \rho_i(g))z_i^*, \text{ for all }g\in G_i.
\end{equation*}

\noindent Until now, we have only worked with $i$. Just as in Claim \ref{grouphomtotal}, we show that the prior equation can be extended to all of $\text{L}(G)$.

\vspace{1mm}

\begin{claim}\label{unit} There exists $g\in G\times G$ and a unitary $c\in \mathbb M_n(\mathbb C)$ such that $z_2^*z_1= u_g \otimes c$. 
\end{claim}

\noindent\textit{Proof of Claim \ref{unit}.} Notice that for every $g\in C \leqslant G_1 \cap G_2$ we have $z_1( u_{\delta_1(g)}\otimes\rho_1(g))z^*_1 = \Delta(u_g)  =z_2(u_{\delta_2(g)}\otimes\rho_2(g))z^*_2$, and hence, if we let $z = z_2^*z_1$ we further have that $z( u_{\delta_1(g)}\otimes\rho_1(g))=(u_{\delta_2(g)}\otimes\rho_2(g))z^*$ for all $g\in C$. From the proof of Claim \ref{grouphomtotal}, we have that every finite index subgroup $C_0\leqslant C$ satisfies $C_{G \times G}(\delta_1(C_0))=\{1\}$. Using Proposition \ref{singlesupport} one can find a unitary $c\in \mathbb M_n(\mathbb C)$ and $(h,l)\in G\times G$ with $z=u_{(h,l)}\otimes c$. 
$\hfill\blacksquare$

\vspace{1mm}

\noindent After conjugating $\delta_2$ by an inner group homomorphism of $G\times G$ and conjugating $\rho_2$ by a unitary in $\mathbb M_n(\mathbb C)$, we obtain the existence a group existence a unitary $z\in \mathcal M \bar\otimes \mathcal M\bar\otimes \mathbb M_n(\mathbb C)$, a monomorphism $\gamma=(\gamma_1,\gamma_2): G \ra G \times G$ and a unitary representation  $\rho: G \ra \mathscr U(\mathbb C^n)$ with

\begin{equation}\label{grouphomeq2}
    \Delta(u_g) = z(u_{\gamma(g)}\otimes \rho(g))z^*, \text{ for all }g\in G.
\end{equation}

\noindent As in Claim \ref{grouprep}, and using that $(\Delta_0\otimes\text{id})\circ\Delta_0=(\text{id}\otimes\Delta_0)\circ\Delta_0$, we obtain that after perturbing $z$ to a new unitary and $\rho$ to a new representation one can show that $\gamma(g)=(g,g)$. 

We finally follow the end of the proof \cite[Theorem 1.3]{CIOS1} to show $n=1$. Let $\mathscr N$ be the von Neumann algebra generated by $\{u_{(g,g)}\otimes x:g\in G, x\in\mathbb M_{n}(\mathbb C)\}$. Observe that $\Delta(\text{L}(G))\subset\mathscr N\subset \Delta(\text{L}(G))\mathbb M_{n}(\mathbb C)$. By Lemma \ref{comultiply}(c), $\mathscr{N}=\Delta(\text{L}(G))$, and so, $1\otimes 1\otimes\mathbb M_n(\mathbb C)\subset \Delta(\text{L}(G))$. From \eqref{grouphomeq2}, $n=1$. Also, $\rho(g)\in\mathbb T$ and $\Delta(u_g)=\rho(g)zu_{(g,g)}z^*$, for all $g\in G$. By \cite[Theorem 3.4]{IPV10}, there exists a unitary $w\in\text{L}(H)$, a character $\mu:G\to\mathbb T$ and a group isomorphism $\rho:G\to H$ with $u_g=\mu(g)wv_{\rho(g)}w^*$. This yields our conclusion in case (b).
\end{proof}

\section{Product and embedding rigidity}\label{SEC:product}

We now present the proofs of the remaining main results, from Theorem \ref{superr} to Corollary \ref{nonembeddingAt1}. For the first two results, which concern product rigidity, we begin by recalling a criterion for the reconstruction of groups from isomorphisms of their von Neumann algebras (Theorem \ref{decompositionofH}).

\subsection{Reconstruction of group direct product splitting under W\texorpdfstring{$^*$}{*} equivalence}

The primary aim of this subsection is to present a criterion (Theorem \ref{decompositionofH}) for reconstructing the direct product structure of a group from its W$^*$-equivalence. This criterion will be used essentially in Section \ref{reconstruction} to derive our main result. The result is implicitly established through the combined proofs of \cite[Theorem 4.3]{CdSS15}, \cite[Theorem 6.1]{CdSS17}, and \cite[Theorem 3.6]{CU18}. Accordingly, rather than reproducing these arguments verbatim, we provide a brief sketch indicating how the criterion follows from the relevant components of those proofs. We encourage the reader to consult the cited works for full details.  

\begin{lem}\label{unionoffinitegroupswithficentralizer} Let $G$ be a countable group and let ${\rm L}(G)$ be the corresponding von Neumann algebra. Assume there exists nonzero projection $p\in {\rm L} (G)$ such that the compression $p{\rm L}(G) p$ is a factor.  Then the FC-radical can be written as $G^{\rm fc}=\bigcup_n F_n$ for an increasing (possible stationary) tower of finite  normal subgroups $(F_n)_{n\in \mathbb N}$ of $G$ satisfying $[G:C_G(F_n)]<\infty$ for all $n\in \mathbb N$. 
\end{lem}

\begin{proof} Letting $z\in {\text L}(G)$ be the central support of $p$ we also have that ${\text L}(G)z$ is a factor. Recall that $G^{fc}=\{ g\in G  :  |\{hgh^{-1}:h\in G\}| \text{ is finite}\}$.  Let $\{\mathcal O_n :  n\in \mathbb N \}$ be an enumeration of all the finite conjugacy orbits of $G$, and consider the normal subgroup generated by  $F_n =\langle \mathcal O_1, \ldots, \mathcal O_n\rangle \lhd G$.  For every $n\in\mathbb N$, notice that $F_n \leqslant F_{n+1}$, $[G:C_G(F_n)]<\infty$ and also $G^{fc}=\bigcup_n F_n$. To get our conclusion it only remains to show that each $F_n$ is finite.  

For every $n\in \mathbb N$ denote the center of $F_n$ by $Z_n=Z(F_n)$.  Since we have $[F_n : Z_n]\leq [G: C_G(F_n)]<\infty$ we only need to show that each $Z_n$ is finite.  Towards this let $C_n =C_G(F_n)$ and consider $D_n =\{ g\in G  :  u_g z\in {\text L}(C_n)z\}$.  Note that $D_n$ forms an intermediate subgroup $C_n\leqslant D_n \leqslant G$. As $[G:C_n]<\infty$, there is a finite set $K\subset G$ of representatives for the left cosets of $D_n\leqslant G$. Consider the map $\mathbb E_n:({\text L}(C_n)z)'\cap \text L(G)z\ra({\text L}(G)z)'\cap \text L(G)z=\mathbb C z$ given by

\begin{equation}\label{formula1}
    \mathbb E_n (x)= [G:D_n]^{-1}\sum_{k\in K} u_k x u_{k^{-1}} \text{ for all }x\in ({\text L}(C_n)z) ' \cap \text L(G)z. 
\end{equation}

\noindent One can easily check that $\mathbb E_n$ is a conditional expectation preserving the trace $\tau_z$. Moreover, using \eqref{formula1}, we see that for every positive operator $x\in ({\text L}(C_n)z)'\cap \text L(G)z$ we have 

\begin{align*}
    \| \mathbb E_n(x)\|^2_{2,z}&=\tau_z(\mathbb E_n (x) \mathbb E_n(x))= \tau_z(\mathbb E_n (x) x)=[G:D_n]^{-1}\sum_k \tau_z (u_k x u_{k^{-1}}x)\geq [G:D_n] ^{-1}\tau_z(x^2). 
\end{align*}

\noindent Since $\|\mathbb E_n(x)\|_{2,z}^2=\|\tau_z(x)z\|_{2,z}^2=\frac{|\tau(x)|^2}{\tau(z)^2}$, from the prior equation we obtain $|\tau(x)|^2\geq \tau(z)[G:D_n]^{-1}\tau(x^2)$, for every $x\geq 0$. Letting $x$ be a nonzero projection in $(\text L(C_n)z)'\cap \text L(G)z$ we obtain $\tau(x)\geq [G:D_n]^{-1}\tau(z)>0$. This implies $(\text L(C_n)z)'\cap \text L(G)z$ must be finite dimensional and since $\text L(Z_n) z\subset  (\text L(C_n)z)'\cap\text L(G)z$, so is $\text L(Z_n) z$. By \cite[Proposition 2.6]{CdSS15}, $Z_n$ is finite. 
\end{proof}

\begin{thm}\label{decompositionofH} Let $G_1,G_2$ be nonamenable icc groups and let $G=G_1\times G_2$. Assume that $H$ is an arbitrary group such that $\M= {\rm L}(G)^t={\rm L}(H)$, for some $t>0$. Moreover, denote by $\Delta:\M \ra \M \bar\otimes \M $ the $\ast$-embedding given by $\Delta(v_h)=v_h\otimes v_h$. Assume the following are satisfied: 

\begin{enumerate}
    \item[(1)] ${\rm L}(G_2)$ is a solid von Neumann algebra;
    \item[(2)] $\Delta({\rm L}(G_1)^t)\prec^s_{\M\oo\M}{\rm L}(G_1)^t\bar\otimes {\rm L}(G_1)^t$;
    \item[(3)] There is no infinite strictly increasing tower of finite subgroups $F_n< H$ such that the intersection $\cap_n C_H(F_n)$ is nonamenable.
\end{enumerate}

\noindent Then $H$ admits a product decomposition $H= H_1 \times H_2$ and there exist a scalar $s>0$ and a unitary $u\in \M$ such that ${\rm L}(G_1)^s= u {\rm L}(H_1) u^*$ and ${\rm L}(G_2)^{t/s}=u{\rm L}(H_2)u^*$. \end{thm}

\begin{proof} From assumption (2), and \cite[Theorem 4.1]{DHI16} (see also \cite{Io11}) there exists a nonamenable group $\Lambda_1\leq H$ with nonamenable centralizer $C_H(\Lambda_1)$ for which ${\text L}(G_1)^t\cong_{\M}^{\text{com}}{\text L}(\Lambda_1)$ (see \cite[Definition 4.1]{CdSS17}). Moreover, by the first part of \cite[Theorem 3.6]{CU18} (until equation (3.16)), there exists a projection $q\in {\text L}(\Lambda_1)$ and $r\in {\text L}(\Lambda_1)'\cap\M$ for which $q({\text L}(\Lambda_1)\vee({\text L}(\Lambda_1)'\cap\M))qr\subseteq qr\M qr$ is an inclusion of finite index II$_1$ factors. In particular, $q{\text L}(\Lambda_1)qr$ and $r({\text L}(\Lambda_1)'\cap\M)rq$ are commuting II$_1$ factors.

Denoting $\Omega=\{\lambda\in H:|\mathscr O_{\Lambda_1}(\lambda)|<\infty\}=vC_H(\Lambda_1)$, \cite[Claim 4.7]{CdSS15} gives that $[H:\Omega \Lambda_1]<\infty$. Since $q{\text L}(\Lambda_1)qr$ is a II$_1$ factor, Lemma \ref{unionoffinitegroupswithficentralizer}, the fact that $C_H(\Lambda_1)$ is nonamenable and the assumption (3) on $H$ give that $\Lambda_1\cap\Omega$ is finite. Now, the proof of \cite[Theorem 3.6 Claims 4.9-4.12]{CdSS15} (see also \cite[Theorem 4.4]{CdSS17}) gives a nonamenable icc subgroup $\Theta\subseteq \Lambda_1$ of finite index for which $[\Theta,\Omega]=\{1\}$ and $[H:\Theta\Omega]<\infty$. Moreover, ${\text L}(G_1)^t\cong_{\M}^{\text{com}}{\text L}(\Theta)$ as ${\text L}(G_1)^t\cong_{\M}^{\text{com}}{\text L}(\Lambda_1)$ and $\Theta\leqslant\Lambda_1$ is finite index.

Finally, using the above together with \cite[Theorem 4.6]{CdSS17}, we obtain a nonamenable icc group $\Omega_1\leqslant C_H(\Theta)$ satisfying ${\text L}(G_2)\cong_{\M}^{\text{com}}{\text L}(\Omega_1)$ and $[H:\Theta\Omega_1]<\infty$. The rest follows by \cite[Theorem 4.7]{CdSS17}.
\end{proof}

\begin{prop}\label{3isnotnecessary} Suppose $G=G_1\times \cdots\times G_n$ is a product of $n$ nonamenable icc groups. Assume that for $2\leq i\leq n$ the groups $G_i$ are amalgamated free products $G_i=\Gamma_{i,1}\ast_{\Sigma_i}\Gamma_{i,2}$ where $\Gamma_{i,k}$ are nonamenable and $\Sigma_i$ is an almost malnormal and infinite index subgroup of $\Gamma_{i,k}$, for $k=1,2$. Let $H$ be an arbitrary group with $\M= {\rm L}(G)^t={\rm L}(H)$, for some $t>0$. Moreover, denote by $\Delta:\M \ra \M \bar\otimes \M $ the $\ast$-embedding given by $\Delta(v_h)=v_h\otimes v_h$. Assume the following are satisfied:

\begin{enumerate}
    \item[(1)] ${\rm L}(G_1)$ is a solid von Neumann algebra, and
    \item[(2)] $\Delta({\rm L}(G_{\hat{1}})^t)\prec^s_{\M\oo\M}{\rm L}(G_{\hat{1}})^t\bar\otimes {\rm L}(G_{\hat{1}})^t$.
\end{enumerate}

\noindent Then $H$ admits a product decomposition $H= H_1 \times H_2$ and there exist a scalar $s>0$ and a unitary $u\in \M$ such that ${\rm L}(G_1)^s= u {\rm L}(H_1) u^*$ and ${\rm L}(G_{\hat{1}})^{t/s}=u{\rm L}(H_2)u^*$.
\end{prop}

\begin{proof} The proof follows by the prior theorem, once we justify that assumption (3) in Theorem \ref{decompositionofH} is satisfied in our case.  

From the prior proof, we obtain a nonamenable group $\Lambda_1\leq H$ with nonamenable centralizer $C_H(\Lambda_1)$ for which ${\text L}(G_{\hat{1}})^t\cong_{\M}^{\text{com}}{\text L}(\Lambda_1)$. Define $\Omega=vC_H(\Lambda_1)$. Notice that the assumption (3) in the prior theorem is only used to show that $\Lambda_1\cap\Omega$ is finite. We follow \cite[Claim 4.8]{CdSS15} to obtain still $|\Lambda_1\cap \Omega|<\infty$. Let $\{\mathcal{O}_i:i\in\mathbb N\}$ be a (countable) enumeration of all the finite orbits of the action by conjugation of $\Lambda_1$ on $H$ and notice $\Omega=\bigcup_i\mathcal{O}_i$. Define $\mathcal{O}_i'=\mathcal{O}_i\cap \Lambda_1$ and notice that $\Lambda_1\cap\Omega=\bigcup_i\mathcal{O}_i'$. For each $k$, let $R_k=\langle\bigcup_{i=1}^k\mathcal{O}_i'\rangle$ and notice it forms an ascending sequence of normal subgroups of $\Lambda_1$ with $\bigcup_kR_k=\Lambda_1\cap \Omega$ and $[\Lambda_1:C_{\Lambda_1}(R_k)]<\infty$. Since $R_k\cap C_{\Lambda_1}(R_k)$ is abelian and $[\Lambda_1:C_{\Lambda_1}(R_k)]<\infty$, it follows that $R_k$ is virtually abelian; and thus, $\Lambda_1\cap \Omega$ is a normal amenable subgroup of $\Lambda_1$.

From \cite[Claim 4.8]{CdSS15}, letting $z=z(pr)$ be the central support of $qr$ in $\text{L}(\Lambda_1)$, we get that $\text{L}(G_{\hat{1}})^s\subset \text{L}(\Lambda_1)z\subset \text{L}(G_{\hat{1}})^{\mu}$ is a finite index inclusion of nonamenable II$_1$ factors, for some $s, \mu>0$. Moreover, we can write $\text{L}(G_{\hat{1}})^{\mu}=(\text{L}(G_{\hat{1}})\bar\otimes\mathbb M_k(\mathbb C))^{t_1/k}$ where $t_1=\tau(qr)\mu$. Take $k$ large enough and denote by $r=t_1/k\in (0,1)$. Fix $2\leq l\leq n$. By \cite[Theorem A]{Va13}, since $\text{L}(G_{\hat{1}})^t\bar\otimes\mathbb M_k(\mathbb C)=(\text{L}(G_{\widehat{1,l}})^t\bar{\otimes} \text{L}(\Gamma_{l,1})\bar\otimes\mathbb M_k(\mathbb C))\ast_{\text{L}(G_{\widehat{1,l}})^t\bar{\otimes}\text{L}(\Sigma_l)\bar\otimes\mathbb M_k(\mathbb C)}(\text{L}(G_{\widehat{1,l}})^t\bar{\otimes}\text{L}(\Gamma_{l,2})\bar \otimes\mathbb M_k(\mathbb C))$, $\text{L}(\Lambda_1\cap \Omega)$ is amenable, and $\text{L}(\Lambda_1)z$ is a factor, we either have:

\begin{enumerate}
    \item[(i)] $\text{L}(\Lambda_1\cap \Omega)z\prec^s\text{L}(G_{\widehat{1,l}})^t\bar{\otimes}\text{L}(\Sigma_l)\bar\otimes\mathbb M_k(\mathbb C)$,
    \item[(ii)] there exists $j=1,2$ with $\text{L}(\Lambda_1)z\prec^s \text{L}(G_{\widehat{1,l}})^t\bar{\otimes}\text{L}(\Gamma_{l,j})\bar\otimes\mathbb M_k(\mathbb C)$, or 
    \item[(iii)] $\text{L}(\Lambda_1)z$ is amenable relative to $\text{L}(G_{\widehat{1,l}})^t\bar{\otimes}\text{L}(\Sigma_l)\bar\otimes\mathbb M_k(\mathbb C)$.
\end{enumerate}

\noindent Since $\text{L}(G_{\hat{1}})^s$ is finite index in $\text{L}(\Lambda_1)z$, (ii) gives $\text{L}(G_{\hat{1}})^s\prec^s\text{L}(G_{\widehat{1,l}})^t\bar{\otimes}\text{L}(\Gamma_{l,j})$, which implies $\Gamma_{l,j}\leqslant G_l$ is a finite index subgroup, a contradiction. Similarly, (iii) gives $\text{L}(G_{\hat{1}})^s$ is amenable relative to $\text{L}(G_{\widehat{1,l}})^t\bar{\otimes}\text{L}(\Sigma_l)\bar\otimes\mathbb M_k(\mathbb C)$. In particular, this means that $G_l$ is amenable relative to $\Sigma_l$, a contradiction. Hence, it must be the case that $\text{L}(\Lambda_1\cap \Omega)z\prec^s\text{L}(G_{\widehat{1,l}})^t\bar{\otimes}\text{L}(\Sigma_l)\bar\otimes\mathbb M_k(\mathbb C)$. If $\text{L}(\Lambda_1\cap \Omega)z\not\prec\text{L}(G_{\widehat{1,l}})^t\bar{\otimes}\mathbb M_k(\mathbb C)$, the fact that $\Sigma_l$ is almost malnormal, $\Lambda_1\cap\Omega$ is normal in $\Lambda_1$, $\text{L}(\Lambda_1)\cong^{\text{cong}}_{\M}\text{L}(G_{\hat{1}})^t$ and \cite[Theorem 4.10]{AMACK25} would give that $\text{L}(G_{\hat{1}})^t\prec \text{L}(G_{\widehat{1,l}})^t\bar{\otimes}\text{L}(\Sigma_l)\bar{\otimes}\mathbb M_k(\mathbb C)$, a contradiction again. Therefore, we must have that $\text{L}(\Lambda_1\cap \Omega)z\prec^s\text{L}(G_{\widehat{1,l}})^t\bar{\otimes}\mathbb M_k(\mathbb C)$. As $l$ was arbitrary, we obtain $\text{L}(\Lambda_1\cap\Omega)z\prec^s\text{L}(G_{\widehat{1,l}})^t\bar\otimes\mathbb M_k(\mathbb C)$ for all $2\leq l\leq n$, and so \cite[Lemma 2.8(2)]{DHI16} implies that $\text{L}(\Lambda_1\cap\Omega)z\prec\mathbb M_k(\mathbb C)$. Using \cite[Proposition 2.6]{CdSS15}, we obtain $\Lambda_1\cap \Omega$ is finite.
\end{proof}


\subsection{W\texorpdfstring{$^*$}{*} and C\texorpdfstring{$^*$}{*} superrigidity for product groups}\label{reconstruction} In this subsection, we establish Theorems \ref{superr} and \ref{superr2}. We begin by proving the following theorem, which provides the main ingredient for both results.

\begin{thm}\label{DeltaGhatjintertwines} For every $i\in \mathbb N$, let $G_i=\Lambda _{k,1}\ast_{\Sigma_i} \Lambda_{k,2}$ be a countable amalgamated free product group where $[\Lambda_{k,i}:\Sigma_i]=\infty$ and $\Sigma_i <\Lambda_{k,i}$ is an almost malnormal subgroup for all $1\leq i\leq 2$, and all $k$. Let $G=\oplus_{i\in\mathbb N}G_i$. Assume that $H$ is an arbitrary group such that $\mathcal{M}= \mathcal{{\rm L}}(G)^t=\mathcal{{\rm L}}(H)$, for some $t>0$. Moreover denote by $\Delta:\mathcal{M} \ra \mathcal{M} \bar\otimes \mathcal{M}$ the $\ast$-embedding given by $\Delta(v_h)=v_h\otimes v_h$. Then for any $j\in \mathbb N$ we have

\begin{equation*}
    \Delta(\mathcal{{\rm L}}(G_{\hat j})^t)\prec^s_{\mathcal{M}\bar{\otimes}\mathcal{M}}\mathcal{{\rm L}}(G_{\hat j})^t\bar\otimes \mathcal{{\rm L}}(G_{\hat j})^t.
\end{equation*}

\noindent Above, for every $s\in \mathbb N$, we denoted by $G_{\hat s}=\{(g_i)_i\in G  |  g_s=1\}$.
\end{thm}

\begin{proof} First, we present a shorter proof for the case when $\Sigma_i$ is amenable. Note that $G_i$ admits an array into a weakly-$\ell^2$ representation that is proper relative to $\{\Lambda_{i,1}, \Lambda_{i,2}\}$. Proceeding as in \cite[Lemma 8.3]{AMCOS25}, the group $G \times G$ admits an array into one of its weakly-$\ell^2$ representations that is proper relative to the family of subgroups $\{  G \times (G_{\hat k} \times \Lambda_{k,s}),  (G_{\hat j} \times \Lambda_{j,t}) \times G \mid s,t = 1,2;\ j,k \in \mathbb{N} \}$. In particular, $G \times G$ is biexact relative to this family of subgroups. Fix $j \in \mathbb{N}$. Since $\Delta(\mathcal{{\text L}}(G_j)^t)$ has no amenable direct summand, \cite[Theorem 15.1.5]{BO08} implies existence of $k \in \mathbb{N}$ and $1 \leq s \leq 2$ such that $\Delta(\mathcal{{\text L}}(G_{\hat j})^t) \prec \mathcal{M}  \bar{\otimes}  \mathcal{{\text L}}(G_{\hat k})^t  \bar{\otimes}  \mathcal{{\text L}}(\Lambda_{k,s})$.

Now assume by contradiction that $\Delta(\mathcal{{\text L}}(G_{\hat j})^t) \nprec \mathcal{M}  \bar{\otimes}  \mathcal{{\text L}}(G_{\hat k})$. Since $\Sigma_k$ is almost malnormal inside $G_k$ (Lemma \ref{malnormal4}), we also have that $\Lambda_{k,s}$ is malnormal inside $G_k$. The same argument as in \cite[Theorem 3.1]{Po03} (see also \cite[Theorem 4.10]{AMACK25}) shows that the normalizer satisfies $\Delta(\mathcal{M}) \subseteq \mathscr{N}_{\mathcal{M}  \bar{\otimes}  \mathcal{M}}(\Delta(\mathcal{{\text L}}(G_{\hat j})^t))'' \prec \mathcal{M}  \bar{\otimes}  \mathcal{{\text L}}(G_{\hat k})^t  \bar{\otimes}  \mathcal{{\text L}}(\Lambda_{k,s})$. Using \cite[Proposition 7.2]{IPV10}, this further implies that $\mathcal{M} \prec_{\mathcal{M}} \mathcal{{\text L}}(G_{\hat k})^t  \bar{\otimes}  \mathcal{{\text L}}(\Lambda_{k,s})$ and hence $[G : G_{\hat k} \times \Lambda_{k,s}] < \infty$, which is a contradiction. Therefore, we must have that $\Delta(\mathcal{{\text L}}(G_{\hat j})^t) \prec \mathcal{M}  \bar{\otimes}  \mathcal{{\text L}}(G_{\hat k})^t$. Since the normalizer satisfies $\mathscr{N}_{\mathcal{M}  \bar{\otimes}  \mathcal{M}}(\Delta(\mathcal{{\text L}}(G_{\hat j})^t))' \cap (\mathcal{M}  \bar{\otimes}  \mathcal{M}) \subseteq \Delta(\mathcal{M})' \cap (\mathcal{M}  \bar{\otimes}  \mathcal{M})$, and the latter relative commutant is trivial by \cite[Proposition 7.2]{IPV10}, from \cite[Lemma 2.4]{DHI16} we get that 

\begin{equation}\label{goodint}
    \Delta(\mathcal {{\text L}}(G_{\hat j})^t)\prec^s_{\mathcal M \bar{\otimes}\M} \M \bar{\otimes} {\text L}(G_{\hat k})^t.
\end{equation}

\noindent Note that if $\zeta$ denotes the flip automorphism on $\mathcal{M}  \bar{\otimes}  \mathcal{M}$, then we have $\zeta \circ \Delta = \Delta$. This, together with the previous inclusion, implies that $\Delta(\mathcal{{\text L}}(G_{\hat j})^t) \prec^s_{\mathcal{M}  \bar{\otimes}  \mathcal{M}} \mathcal{{\text L}}(G_{\hat k})^t  \bar{\otimes}  \mathcal{M}$, and using \cite{DHI16} gives $\Delta(\mathcal{{\text L}}(G_{\hat j})^t) \prec^s_{\mathcal{M}  \bar{\otimes}  \mathcal{M}} \mathcal{{\text L}}(G_{\hat k})^t  \bar{\otimes}  \mathcal{{\text L}}(G_{\hat k})^t$. The same argument as in \cite[Theorem 4.3]{Dr20}, allows us to pick $k=j$.

\vspace{2mm}

Next, we present a proof in the general case, based on different arguments. Fix $j,k\in\mathbb N$. Write $\M \bar\otimes \M$ as an amalgamated free product $\M\bar{\otimes}\M= (\M \bar{\otimes}\M_{\hat{k}}\bar{\otimes} {\text L}(\Lambda_{k,1}))\ast_{\M\bar{\otimes}\M_{\hat{k}} \bar{\otimes} {\text L}(\Sigma_k)}(\M \bar{\otimes}\M_{\hat{k}}\bar{\otimes} {\text L}(\Lambda_{k,2}))$, for $\M_{\hat{k}}=\text{L}(G_{\hat{k}})^t$. Notice that $\Delta({\text L}(G_{\hat j})^t),  \Delta({\text L}(G_j)^t)$ are commuting nonamenable von Neumann subfactors of $\M\bar{\otimes}\M$. By Theorem \ref{commutlocinamalgams} we obtain one of the following must hold:

\begin{enumerate}
    \item[(1)] $\Delta({\text L}(G_{\hat j})^t)\prec \M \bar{\otimes} {\text L}(G_{\hat k})^t \bar{\otimes} {\text L}(\Sigma_k)$,
    \item[(2)] $\Delta({\text L}(G_j))\prec\M \bar{\otimes} {\text L}(G_{\hat k})^t \bar{\otimes} {\text L}(\Sigma_k)$, 
    \item[(3)] $\Delta(\M)=\Delta({\text L}(G_{\hat j})^t)\vee \Delta({\text L}(G_j))\prec \M \bar{\otimes} {\text L}(G_{\hat k})^t\bar{\otimes} {\text L}(\Lambda_{k,s})$, for some $1\leq s\leq 2$, or
    \item[(4)] $\Delta(\M)=\Delta({\text L}(G_{\hat j})^t)\vee \Delta({\text L}(G_j))$ is amenable relative to $\M \bar{\otimes} {\text L}(G_{\hat k})^t \bar{\otimes} {\text L}(\Sigma_k)$ inside $\M \bar{\otimes} \M$.
\end{enumerate}

\noindent If (3) holds then \cite[Proposition 7.2]{IPV10} implies that $[G:G_{\hat{k}} \times \Lambda_{k,s}]<\infty$ and hence $[G_k: \Lambda_{k,s}]<\infty$, a contradiction. If (4) holds then \cite[Theorem 4.1(b)]{BV13} implies that ${\text L}(G)$ is amenable relative to {\text L}$(G_{\hat k}\times \Sigma_k)$. Therefore, $G_k$ is amenable relative to $\Sigma_k$ inside $G_k$, a contradiction. 

Assume (1) holds and assume by contradiction that $\Delta({\text L}(G_{\hat j})^t)\not\prec \M \bar{\otimes} {\text L}(G_{\hat k})^t$. Since $\Sigma_k$ is almost malnormal the same argument as in \cite[Theorem 3.1]{Po03} (see also \cite[Theorem 4.10]{AMACK25}) shows that $\Delta(\M)\subseteq \mathscr N_{\M\bar{\otimes}\M}(\Delta({\text L}(G_{\hat j})^t))''\prec\M \bar{\otimes} {\text L}(G_{\hat k})^t \bar{\otimes} {\text L}(\Sigma_k)$. Using \cite[Proposition 7.2]{IPV10} we obtain $\M \prec_\M {\text L}(G_{\hat k})^t\bar{\otimes}{\text L}(\Sigma_k)$ and hence $[G: G_{\hat k}\times \Sigma_k]<\infty$, which is a contradiction. In conclusion, (1) gives that $\Delta({\text L}(G_{\hat j})^t)\prec \M \bar{\otimes} {\text L}(G_{\hat k})^t$. Since the normalizer satisfies $\mathscr N_{\M\bar{\otimes}\M}(\Delta({\text L}(G_{\hat j})^t))'\cap (\M \bar{\otimes}\M) \subseteq \Delta (\M)'\cap (\M \bar{\otimes}\M) $ and the last relative commutant is trivial by \cite[Proposition 7.2]{IPV10}, using \cite[Lemma 2.4]{DHI16} we further get that

\begin{equation*}
    \Delta({\text L}(G_{\hat j})^t)\prec^s_{\M \bar{\otimes}\M} \M \bar{\otimes} {\text L}(G_{\hat k})^t.
\end{equation*}

\noindent Note that if $\zeta$ is the flip automorphism on $\M \bar{\otimes}\M$ then we have $\zeta \circ \Delta =\Delta$. This together with the prior equation gives $\Delta({\text L}(G_{\hat j})^t)\prec^s_{\M \bar{\otimes}\M}  {\text L}(G_{\hat k})^t\bar{\otimes}\M$. Therefore, $\Delta({\text L}(G_{\hat j})^t)\prec^s_{\M \bar{\otimes}\M}  {\text L}(G_{\hat k})^t\bar{\otimes}{\text L}(G_{\hat k})^t$, as wanted.

Proceeding in a similar manner one can show that (2) also gives 

\begin{equation}\label{badint}
    \Delta({\text L}(G_{j}))\prec^s_{\M \bar{\otimes}\M} \M \bar{\otimes} {\text L}(G_{\hat k})^t.
\end{equation}

\noindent Assume \eqref{badint} holds for every $k\in\mathbb N$. By \cite[Lemma 2.6]{DHI16}, $\Delta({\text L}(G_j))\prec^s_{\M\bar{\otimes}\M}\M\bar{\otimes}{\text L}(G_{\widehat{F}})^t$, for every finite set $F\subset\mathbb N$. By \cite[Claim 3.4]{CU18}, it follows that $\Delta({\text L}(G_j))$ is amenable relative to $\M\bar{\otimes}1$, and by \cite[Theorem 7.2(4)]{IPV10} ${\text L}(G_j)$ is amenable, a contradiction.

Hence, we have shown that for every $j\in\mathbb N$ there exists $k\in\mathbb N$ for which $\Delta({\text L}(G_{\hat{j}})^t)\prec_{\M\bar{\otimes}\M}{\text L}(G_{\hat{k}})^t\bar{\otimes}{\text L}(G_{\hat{k}})^t$ holds. Similarly as before, \cite[Theorem 4.3]{Dr20} allows us to pick $k=j$.
\end{proof}

\begin{thm}\label{productforamalgamated} For every $i\in \mathbb N$, let $G_i=\Lambda_{i,1}\ast_{\Sigma_i} \La_{i,2}$ be a countable amalgamated free product where $[\La_{i,k}:\Sigma_k]=\infty$  and  $\Sigma_i <\La_{i,k}$
is an almost malnormal amenable subgroup for all $1\leq i\leq 2$. 
Assume that ${\rm L}(\La_{i,j})$ is solid for any $i\in \mathbb N,\: 1\leq j\leq 2$.
Let $G=\oplus_{i\in\mathbb N}G_i$. Assume that $H$ be a torsion free group or more generally a group for which there is $C\in \mathbb N$ such that any finite subgroup $K<H$ satisfies $|K|\leq C$. Assume that $\M= {\rm L}(G)^t={\rm L}(H)$, for some $t>0$. 

Then $H$ admits a direct sum decomposition $H=(\oplus_{i\in \mathbb N} H_i) \times A$, where $H_i$ is an icc group for all $i\in \mathbb N$ and $A$ is an icc amenable group. Moreover, there is a sequence $(t_i)_{i\in \mathbb N}\subset \mathbb R_+^*$ such that for every $n\in \mathbb N$ one can find a unitary $u_n\in {\rm L}(H)$ such that 

\begin{equation*}
    \begin{split}& u_n{\rm L}(G_j)^{t_j}u_n^*={\rm L}(H_j)\quad \text{ for all } j\leqslant  n \text{; and }\\& u_n{\rm L}(\oplus_{j> n} G_j)^{t/(t_1\cdots t_{n+1})}u_n^*={\rm L}((\oplus_{j>n} H_j) \times A).\end{split}
\end{equation*} 
\end{thm}

\begin{proof} Let $\Delta:{\text L}(H)\to{\text L}(H)\oo{\text L}(H)$ be the $*$-homomorphism given by $\Delta(v_h)=v_h\otimes v_h$, for $h\in H$, and consider it as a map $\Delta:{\text L}(G)^t\to \text L(G)^t\oo{\text L}(G)^t$. Consider ${\text L}(G)^t={\text L}(G_{\hat{1}})^t\oo{\text L}(G_1)$, where $\text{L}(G_1)$ is solid. By Theorem \ref{DeltaGhatjintertwines}, $\Delta(\text{L}(G_{\hat{i}})^t)\prec_{\M\oo\M}{\text L}(G_{\hat{i}})^t\oo{\text L}(G_{\hat{i}})$, and from Theorem \ref{decompositionofH} we obtain commuting, nonamenable, icc subgroups $H_1, K_1\leqslant H$, $t_1>0$ and a unitary $u_1\in\M$ for which $u_1{\text L}(G_1)^{t_1}u_1^*={\text L}(H_1)$ and $u_1{\text L}(G_{\hat{1}})^{t/t_1}u_1^*={\text L}(K_1)$. Applying Theorem \ref{decompositionofH} again in the last relation for the group $G_{\hat{1}}$, we obtain a decomposition $K_1=H_2\oplus K_2$, $u_2\in \mathscr U(u_1{\text L}(G_{\hat 1}^{t/t_1})u_1^*)$ and $t_2>0$ such that $u_2{\text L}(G_2)^{t_2}u_2^*={\text L}(H_2)$ and $u_2{\text L}(G_{\widehat{\{1,2\}}})^{t/(t_1t_2)}u_2^*={\text L}(K_2)$. Proceeding inductively one has $K_{n-1}=H_n\oplus K_n$, a unitary $u_n\in\mathscr U(u_{n-1}{\text L}(G_{\mathbb N\setminus\overline{1,n-1}})^{t/(t_1\cdots t_{n-1})}u_{n-1}^*)$ and $t_n>0$ such that $u_n{\text L}(G_n)^{t_n}u_n^*={\text L}(H_n)$ and $u_n{\text L}(G_{\mathbb N\setminus\overline{1,n}})^{t/(t_1\cdots t_{n})}u_n^*={\text L}(K_n)$. Therefore, $K_{n+1}\leqslant K_n$ and $H=\oplus_nH_n\oplus A$, where $A=\cap_nK_n$. 

Letting $v_n=u_1u_2\cdots u_n$ we see that $v_n{\text L}(G_i)^{t_i}v_n^*={\text L}(H_i)$ for all $1\leq i\leq n$, and $v_n{\text L}(G_{\mathbb N\setminus\overline{1,n}})^{t/(t_1\cdots t_n)}v_n^*={\text L}((\oplus_{i\geq n+1}H_i)\times A)$. By \cite[Proposition 2.7]{CU18}, $A$ is icc and amenable.  
\end{proof}

\vspace{1mm}


\begin{proof}[Proof of Theorem \ref{superr}] Let $\Delta:{\text L}(H)\to{\text L}(H)\oo{\text L}(H)$ be the $*$-homomorphism given by $\Delta(v_h)=v_h\otimes v_h$, for $h\in H$, and consider it as a map $\Delta:{\text L}(G)^t\to \text L(G)^t\oo{\text L}(G)^t$, where $0<t\leq 1$. Observe that ${\text L}(G)^t={\text L}(G_{\hat{1}})^t\oo{\text L}(G_1)$, where $\text{L}(G_1)$ is solid. By Theorem \ref{DeltaGhatjintertwines}, $\Delta(G_{\hat{i}})^t\prec_{\M\oo\M}{\text L}(G_{\hat{i}})^t\oo{\text L}(G_{\hat{i}})^t$, and from Proposition \ref{3isnotnecessary}, we obtain commuting, nonamenable, icc subgroups $H_1, K_1\leqslant H$, $t_1>0$ and a unitary $u_1\in\M$ for which $u_1{\text L}(G_1)^{t_1}u_1^*={\text L}(H_1)$ and $u_1{\text L}(G_2\oplus\cdots\oplus G_k)^{t/t_1}u_1^*={\text L}(K_1)$. Working with $u_1{\text L}(G_2\oplus\cdots\oplus G_n)^{t/t_1}u_1^*={\text L}(K_1)$, and by induction, one has $K_{n-2}=H_{n-1}\oplus K_{n-1}$, a unitary $u_{n-1}\in\mathscr U(u_{n-2}{\text L}(G_{n-2})^{t/(t_1\cdots t_{n-2})}u_{n-2}^*)$ and $t_{n-1}>0$ such that $u_{n-1}{\text L}(G_{n-1})^{t_{n-1}}u_{n-1}^*={\text L}(H_{n-1})$ and $u_{n-1}{\text L}(G_{n})^{t/(t_1\cdots t_{n-1})}u_{n-1}^*={\text L}(K_{n-1})$. Denote by $H_n=K_{n-1}$ and $t_n=t/(t_1\cdots t_{n-1})$. Since $\text{L}(\oplus H_i)=\text{L}(H)$, it follows that $H=\oplus_{i=1}^nH_i$.

Letting $w=u_1u_2\cdots u_{n-1}$ we see that $w{\text L}(G_i)^{t_i}w^*={\text L}(H_i)$ for all $1\leq i\leq n$, and $t_1\cdots t_n=t\leq 1$. It follows that there exists $1\leq i\leq n$ for which $t_i\leq 1$. For such $i$, $w\text{L}(G_i)^{t_i}w^*=\text{L}(H_i)$. Since $G_i$ is W$^*$-superrigid, $t_i=1$ and $G_i\cong H_i$. Thus, $t_1\cdots t_{i-1}t_{i+1}\cdots t_n=t$. Continuing by induction we see that $t_i=1$, for each $i$, so that $t=1$ and $G_i\cong H_i$. From $w\text{L}(G_i)w^*=\text{L}(H_i)$, for $1\leq i\leq n$, and Theorem \ref{ATrigid}, we see that there exists a character $\eta_i:G_i\to\mathbb T$ and a group isomorphism $\delta_i:G_i\to H_i$ for which $u_iu_gu_i^*=\eta_i(g)v_{\delta_i(g)}$, for all $g\in G_i$. Taking $\delta=\oplus_i\delta_i$ and $\eta=\prod_i\eta_i$, we see that $\delta:G\to H$ is a group isomorphism satisfying $wu_{g}w^*=\eta(g)v_{\delta(g)}$, for all $g\in G$. This shows that the map $\text{L}(G)\cong \text{L}(H)$ is group-like.
\end{proof}

\noindent To derive our applications to C$^*$-superrigidity we need the following elementary result.

\begin{prop}\label{extendingembedding}
Let $G$ and $H$ be discrete countable groups such that their reduced $C^*$-algebras $C_r^*(G)$ and $C_r^*(H)$ have a unique trace. Then any $\ast$-embedding $\theta \colon C_r^*(G) \to C_r^*(H)$ extends to a $\ast$-embedding $\tilde{\theta} \colon \mathrm{L}(G) \to \mathrm{L}(H)$. In particular, if $\theta$ is a $\ast$-isomorphism, then so is $\tilde{\theta}$.
\end{prop}

\begin{proof}
Let $\tau \colon C_r^*(G) \to \mathbb{C}$ and $\rho \colon C_r^*(H) \to \mathbb{C}$ denote the canonical traces. Since $C_r^*(G)$ has a unique trace, we have $\rho \circ \theta = \tau$. Define a map $U_0 \colon \ell^2G \to \ell^2H$ by $U_0(x\delta_e) = \theta(x)\delta_e$ for all $x \in C_r^*(G)$. For any $x,y \in C_r^*(G)$, we compute

\begin{align*}
    \langle U_0(x\delta_e), U_0(y\delta_e) \rangle
    &= \langle \theta(x)\delta_e, \theta(y)\delta_e \rangle = \langle \theta(y^*x)\delta_e, \delta_e \rangle = \rho(\theta(y^*x)) = \tau(y^*x) = \langle x\delta_e, y\delta_e \rangle.
\end{align*}

\noindent Thus $U_0$ extends by linearity to an isometry
$U \colon \ell^2G \to \ell^2H$. Define $\tilde{\theta} \colon \mathbb B(\ell^2G) \to \mathbb B(\ell^2H)$ by
$\tilde{\theta}(x) = U x U^*$. Since $U$ is an isometry, $\tilde{\theta}$ is a $\ast$-embedding. Let $\mathcal{H} = U(\ell^2G)$, which is a Hilbert subspace of $\ell^2H$. We claim that for every $x \in C_r^*(G)$ we have $\tilde{\theta}(x)\big|_{\mathcal{H}} = \theta(x)\big|_{\mathcal{H}}$. Indeed, fix $\eta \in \mathcal{H}$ and choose a net $\{x_i\} \subset C_r^*(G)$ such that $\theta(x_i)\delta_e \to \eta$. Then

\begin{align*}
    \tilde{\theta}(x)\eta = U x U^* \eta = \lim_i U x U^* \theta(x_i)\delta_e = \lim_i U x x_i \delta_e = \lim_i \theta(x x_i)\delta_e = \theta(x)\eta,
\end{align*}

\noindent as claimed. Finally, since $U$ is an isometry, a net $\{x_i\} \subset \mathbb B(\ell^2G)$ converges to $x$ in the weak operator topology if and only if $\{U x_i U^*\}$ converges to $U x U^*$ in the weak operator topology. It follows that $\tilde{\theta}(\mathrm{L}(G)) = \theta(C_r^*(G))'' \subseteq \mathrm{L}(H)$. Therefore, $\tilde{\theta}$ extends $\theta$ to a $\ast$-embedding $\mathrm{L}(G) \hookrightarrow \mathrm{L}(H)$.
\end{proof}

\begin{lem}\label{productCamalgamated} For every $i\in \mathbb N$, let $G_i=\Lambda_{i,1}\ast_{\Sigma_i} \La_{i,2}$ be an amalgamated free product where $\Sigma_i$ is almost malnormal in $G_i$ and  $\La_{i,k}\in {\mathcal WR}(A_{i,k}, B_{i,k})$ with $A_{i,k}$ abelian torsion free and $B_{i,k}$ torsion free and hyperbolic relative to residually finite subgroup with Haagerup property. Let $G=\oplus_{i\in\mathbb N}G_i$. Assume that $H$ be any group such that $ C^*_r(G)\cong C^*_r(H)$. Then $H$ is torsion free and has trivial amenable radical. 
\end{lem}

\begin{proof} First we argue that $H$ is a torsion free group. Since $B_i$ is a subgroup of a hyperbolic group it satisfies the Baum-Connes conjecture \cite{La12}. Since $A_{i,k}$ is Haagerup then so is $A_{i,k}^{(B_{i,k})}$. Therefore using \cite{O-O01b} we conclude that $\Lambda_{i,k}$ satisfies the Baum-Connes conjecture for all $i\in\mathbb N$, $k=1,2$. Moreover, by \cite[Corollary 1.2]{O-O01b}, $G_i$ also satisfies Baum-Connes conjecture. Since this passes to direct sums \cite[Theorem 8.7(2)(d)]{Lu18}, it follows that $G$ satifies the Baum-Connes conjecture. As $G$ is torsion free it satisfies the Kaplansky conjecture, i.e.\ $C^*_r(G)$ has no nontrivial projections. Hence, so does $C^*_r(H)\cong C^*_r(G)$, and thus $H$ is torsion free.

Fix $i\in \mathbb N$. Next we argue that $G_i$ has trivial amenable radical. Let $A \lhd G_i$ be an amenable normal subgroup and let $\mathcal A :={\rm L}(A)\subset {\rm L}(G_i)$. Using \cite[Theorem A]{Va13} we get that $\mathcal A \prec_{{\rm L}(G_i)} {\rm L}(\Sigma_i)$. Since $\mathcal A \subseteq {\rm L}(G_i)$ is regular, by \cite[Proposition 8]{HPV11} we get that $\mathcal A \prec {\rm L}\left (\cap_k (g_k \Sigma_i g_k^{-1})\right )$ for every finite family of elements $g_1,\ldots,g_l \in G_i$. The malnormality of $\Sigma_i$ inside $G_i$ implies that $\mathcal A \prec \mathbb C1$. Using \cite[Proposition 2.6]{CdSS15} we get that $A$ is finite. However since $G_i$ is icc we must have that $A=\{1\}$. Now let $A\lhd G$ be  an amenable normal subgroup. Fix $i\in \mathbb N$ and let $\pi_i : G\ra G_i$ be the canonical projection. Observe that $\pi_i(A)$ is a normal amenable subgroup of $G_i$ and by the previous part we get $\pi_i(A)=\{1\}$ so that $A\leqslant \ker(\pi_i)=\oplus_{j\neq i} G_j$. Since this holds for all $i$, we have $A\leqslant \cap_{i\in \mathbb N} \ker(\pi_i)=\{1\}$. Therefore, $G$ has trivial amenable radical. By \cite{BKKO14},  $C^*_r(G)$ has unique trace and thus so does $C^*_r(H)$. Using \cite{BKKO14} again, we get $H$ has trivial amenable radical.\end{proof}

\begin{proof}[Proof of Theorem \ref{superr2}] By Lemma \ref{productCamalgamated}, $H$ is torsion free and has trivial amenable radical, and thus, by Proposition \ref{extendingembedding}, the isomorphism $C_r^*(G)\cong C_r^*(H)$ can be extended to an isomorphism $\text{L}(G)\cong\text{L}(H)$ (see also \cite{BKKO14}). By Theorem \ref{productforamalgamated}, $H$ admits a direct sum decomposition $H=(\oplus_{i\in \mathbb N} H_i) \oplus A$, where $H_i$ is an icc group for all $i\in \mathbb N$ and $A$ is an icc amenable group. Moreover, there is a sequence $(t_i)_{i\in \mathbb N}\subset \mathbb R_+^*$ such that for every $n\in \mathbb N$ one can find a unitary $u_n\in {\text L}(H)$ such that 

\begin{equation*}
    \begin{split}& u_n{\text L}(G_j)^{t_j}u_n^*={\text L}(H_j)\quad \text{ for all } j\leqslant  n \text{; and }\\& u_n{\text L}(\oplus_{j> n} G_j)^{1/(t_1\cdots t_{n+1})}u_n^*={\text L}((\oplus_{j>n} H_j) \times A).\end{split}
\end{equation*} 

\noindent Since $A$ is amenable, and $H$ has trivial amenable radical, $A=\{1\}$. Moreover, since each $H_j$ is torsion free, Theorem \ref{super+torsion}(b) gives that $t_j=1$ and $G_j\cong H_j$; and hence, $G\cong H$. The remaining part follows similarly as in the end of the proof of Theorem \ref{superr}.
\end{proof}

\subsection{Automorphisms and embeddings}\label{SEC:embedding} In this part, we explain how our previous results and methods enable us to provide examples of nonamenable groups whose reduced C$^*$-algebras have only weakly inner outer automorphisms. A similar result for embeddings of such algebras is also obtained.


\begin{proof}[Proof of Theorem \ref{ATembeddings}] Since $G$ has trivial amenable radical (by Theorem \ref{superr2}), and using Proposition \ref{extendingembedding}, we can extend $\theta:C_r^*(G)\to C_r^*(G)$ to a $*$-embedding $\tilde{\theta}:\text{L}(G)\to\text{L}(G)$. Now, notice that the proof of Theorem \ref{superr} treats the comultiplication on $\text{L}(G)$ as an embedding until Claim \ref{j=k}. Thus, the same methods apply here verbatim and for each $i=1,2$, we obtain an injective group homomorphism $\delta_i:G_i\to G$, a character $\mu_i:G_i\to\mathbb T$ and a unitary $z_i\in\mathscr U(\text{L}(G))$ such that $\tilde\theta(u_g)=\mu_i(g)z_iu_{\delta_i(g)}z_i^*$, for all $g\in G_i$. The proof of Claim \ref{grouphomtotal}, gives an injective group homomorphism $\delta:G\to G$, a character $\eta:G\to\mathbb T$ and a unitary $z\in\mathscr{U}(\text{L}(G))$ satisfying $\tilde\theta(u_g)=\eta(g)zu_{\delta(g)}z^*$, for every $g\in G$.
\end{proof}

\noindent When the group $G$ is assumed to further be contained in the subclass $\mathcal{AT}_0$, we immediately obtain the conclusion of Corollary \ref{AT0inner}, as every injective homomorphism $\delta:G\to G$ is an inner automorphism (from Theorem \ref{afpwithnoend}). 

\begin{proof}[Proof of Corollary \ref{nonembeddingAt1}.] Let $G,H\in\mathcal{AT}_1$ be non-isomorphic groups, and assume that $\theta:C_r^*(G)\to C_r^*(H)$ is an embedding. By Theorem \ref{extendingembedding}, $\theta$ can be extended to an embedding between the von Neumann algebras $\text{L}(G)$ and $\text{L}(H)$, which we still denote by $\theta$. Similarly as in the proof of Theorem \ref{ATembeddings}, we can use the proof of Theorem \ref{superr} (since $H\in\mathcal{AT}$) to obtain injective group homomorphisms $\delta_i:G_i\to H$, a character $\mu_i:G_i\to\mathbb T$ and a unitary $z_i\in\mathscr U(\text{L}(H))$ with $\theta(u_g)=\mu_i(g)z_iv_{\delta_i(g)}z_i^*$, for all $g\in G_i$. Applying Claim \ref{grouphomtotal} gives an injective group homomorphism $\delta:G\to H$, contradicting Theorem \ref{AT1nonembedding}.
\end{proof}

\begin{rem} The previous results also enables us to provide examples of icc groups $G$ such that ${\rm L}(G)$ is McDuff and has the property that for every automorphism $\theta\in \text{Aut}(C_r^*(G))$ there is a sequence of unitaries $w_n\in {\rm L}(G)$ such that

\begin{equation}
    \|\theta(x)w_n -w_n x\|_\infty\ra 0 \text{ for all } x\in C_r^*(G).
\end{equation}

\noindent Indeed just take $G=\oplus_{\in \mathbb N} G_i$, where $G_i \in \mathcal{AT}$ are pairwise nonisomorphic groups with trivial outer automorphism group and trivial abelianization. We leave the details to the reader.
\end{rem}

\end{document}